\documentclass[fleqn,10pt]{article}
\usepackage{amsmath,amssymb,amsthm,xcolor,framed,esint,dsfont}
\usepackage[english]{babel}
\usepackage[margin=3cm]{geometry}
\usepackage{csquotes}
\usepackage[shortlabels]{enumitem}
\usepackage[title]{appendix}
\def\Xint#1{\mathchoice
	{\XXint\displaystyle\textstyle{#1}}%
	{\XXint\textstyle\scriptstyle{#1}}%
	{\XXint\scriptstyle\scriptscriptstyle{#1}}%
	{\XXint\scriptscriptstyle\scriptscriptstyle{#1}}%
	\!\int}

\def\XXint#1#2#3{{\setbox0=\hbox{$#1{#2#3}{\int}$}
		\vcenter{\hbox{$#2#3$}}\kern-.5\wd0}}

\newcommand{\norm}[1]{{\Vert #1\Vert}}
\newcommand{\abs}[1]{{\left\vert #1\right\vert}}
\newcommand{\R}{{\mathbb R}}

\newcommand{\nl}{\newline}

\newcommand{\ep}{\varepsilon}
\newcommand{\lt}{\left}
\newcommand{\rt}{\right}
\newcommand{\na}{\nabla}
\newcommand{\nn}{\nonumber}
\newcommand{\one}{{\mathds{1}}}
\newcommand{\e}{\varepsilon}
\newcommand{\la}{\langle}
\newcommand{\ra}{\rangle}
\DeclareMathOperator{\dv}{div}

\newcommand{\loc}{{\rm loc}}
\newcommand{\qd}{\quad}

\newcommand{\Z}{\mathbb{Z}}
\newcommand{\PPI}{\mathcal{P}}

\newcommand{\AI}{\mathcal{A}}
\newcommand{\BI}{\mathcal{B}}

\newcommand{\FI}{\mathcal{F}}

\newcommand{\RI}{\mathcal{R}}
\newcommand{\WI}{\mathcal{W}}

\newcommand{\MI}{\mathcal{Q}}

\newcommand{\wt}{\widetilde}
\newcommand{\ti}{\tilde}
\newcommand{\hi}{\tilde{h}}

\newcommand{\blue}[1]{{\textcolor{blue}{#1}}}

\newtheorem{thm}{Theorem}
\newtheorem{prop}[thm]{Proposition}
\newtheorem{lem}[thm]{Lemma}
\newtheorem{cor}[thm]{Corollary}
\newtheorem{con}[thm]{Conjecture}
\newtheorem{defin}[thm]{Definition}
\newtheorem{rem}[thm]{Remark}
\newcommand{\bp}[1]{{\textcolor{black}{#1}}}

\title{}
\date{}
\author{}

\begin{document}

\title{Factorization for entropy production of the Eikonal equation and regularity}
\date{}
\author{Andrew Lorent\footnote{Department of Mathematical Sciences, University of Cincinnati, Cincinnati, OH 45221, USA. Email: lorentaw@uc.edu}
	\and Guanying Peng\footnote{Department of Mathematical Sciences, Worcester Polytechnic Institute, Worcester, MA 01609, USA. Email: gpeng@wpi.edu}}
\maketitle

\maketitle
\begin{abstract}
	The Eikonal equation arises naturally in the limit of the second order Aviles-Giga functional whose $\Gamma$-convergence is a long standing challenging problem. The theory of entropy solutions of the Eikonal equation plays a central role in the variational analysis of this problem. Establishing fine structures of entropy solutions of the Eikonal equation, e.g. concentration of entropy measures on $\mathcal{H}^1$-rectifiable sets in $2$D, is arguably the key missing part for a proof of the full $\Gamma$-convergence of the Aviles-Giga functional. In the first part of this work, for $p\in \lt(1,\frac{4}{3}\rt]$  we establish an $L^p$ version of the main theorem of \cite{GL}. Specifically  we show that if $m$ is a solution to the Eikonal equation, then $m\in B^{\frac{1}{3}}_{3p,\infty,\loc}$  is equivalent to all entropy productions of $m$ being in $L^p_{\loc}$.  Given the main result of \cite{GL}, this result also shows that as a consequence of a  weak form of the Aviles-Giga conjecture (namely the conjecture that all solutions to the Eikonal equation whose  entropy productions are in $L^p_{\loc}$ are rigid) - the rigidity/flexibility threshold of the Eikonal equation is exactly the space $ B^{\frac{1}{3}}_{3,\infty,\loc}$. In the second part of this paper, under the assumption that all entropy productions are in $L^p_{\loc}$, we establish a factorization formula for entropy productions of solutions of the Eikonal equation in terms of the two Jin-Kohn entropies. A consequence of this formula is control of all entropy productions by the Jin-Kohn entropies  in the $L^p$ setting - this is a strong extension of the main result of \cite{LP}.
\end{abstract}

\section{Introduction}  
\subsection{The Aviles-Giga functional and the Eikonal equation}
 The Aviles-Giga functional is a second order functional that (subject to appropriate boundary conditions) models phenomena from thin film blistering to smectic liquid crystals, and is also the most natural higher order generalization of the Cahn-Hilliard functional. It is defined as
\begin{equation*}
AG_{\ep}(u)=\int_{\Omega}\lt(\ep \lt|\na^2 u\rt|^2  +\frac{\lt(1-\lt|\na u\rt|^2\rt)^2}{\ep}\rt)\; dx
\end{equation*}
for $u\in W^{2,2}(\Omega)$ over a bounded domain $\Omega\subset\R^2$, where $\nabla^2 u$ is the Hessian matrix of the scalar-valued function $u$ and $\ep>0$ is a small parameter. The Aviles-Giga conjecture for the $\Gamma$-limit of $AG_{\ep}$ is one of the central conjectures in the theory of $\Gamma$-convergence and has attracted a great deal of attention, yet remains open; see for example  \cite{avilesgig,avgig1,ADM,mul2,ottodel1, CD, ark, GL}. What makes the Aviles-Giga conjecture much more challenging than the $\Gamma$-convergence of the Cahn-Hilliard functional is the cubic power scaling in the former, which makes the $BV$ function theory inapplicable.

One of the foundational theorems established for the Aviles-Giga functional is the compactness \cite{ADM,mul2}. Specifically, given a sequence $\{u_\ep\}\subset W^{2,2}(\Omega)$ such that $\sup_\ep AG_{\ep}(u_\ep)<\infty$, it has a subsequence that converges strongly in $W^{1,3}$ to some limiting function $u$ (here we are stating the compactness result in \cite{ADM}; the version proved in \cite{mul2} is slightly different). The limiting function $u$ must satisfy the Eikonal equation given by
\begin{equation*}
\lt|\na u\rt|=1 \qd\text{ a.e in }\Omega. 
\end{equation*}
In two dimensions, the above Eikonal equation can be equivalently formulated in terms of vector fields $m:\Omega\to\R^2$ as 
\begin{equation}
	\label{ageq3}
	|m|=1\text{ a.e.},\;\;\;\;\; \dv m=0\text{ in }\mathcal{D}'(\Omega)
\end{equation}
by identifying $m=\na^{\perp}u$. This formulation of the Eikonal equation in $2$D is enlightening in that one can view \eqref{ageq3} as a scalar conservation law in $1$D. This was first observed by the authors of \cite{mul2}, who introduced the concept of \em entropies \em as a central tool for the analysis of the Aviles-Giga functional (the implicit use of the concept of entropies in this setting appeared in \cite{JK,ADM}). In \cite{mul2}, entropies for the Eikonal equation \eqref{ageq3} are defined as vector fields $\Phi\in C^{\infty}_c(\R^2;\R^2)$ such that $\dv\Phi(m)\equiv 0$ if $m$ is a smooth solution to (\ref{ageq3}). Such entropies can be characterized explicitly; see \eqref{harment}. This is completely analogous to entropies for hyperbolic conservation laws. For solutions $m=\na^\perp u=\lim_{\ep\rightarrow 0} \na u_\e^{\perp}$ with $\sup_{\ep}AG_{\ep}(u_\ep)<\infty$, it can be shown that $\dv\Phi(m)$ are finite Radon measures, called \emph{entropy measures}, and (if $m$ has the additional property that $m\in BV$) they detect the jumps in $m$. As such, the function space for the $\Gamma$-convergence of the Aviles-Giga functional is a subset of the space $\AI(\Omega)$ consisting of \emph{entropy solutions} of the Eikonal equation, i.e. weak solutions $m$ to \eqref{ageq3} such that $\dv\Phi(m)\in\mathcal{M}(\Omega)$ for all entropies $\Phi$, where $\mathcal{M}(\Omega)$ is the set of finite Radon measures on $\Omega$. It is thus natural to understand the space $\AI(\Omega)$ for the purpose of the full proof  of the $\Gamma$-convergence of the Aviles-Giga functional.

The lack of understanding of fine structures of the space $\AI(\Omega)$ constitutes one of the major obstacles in the study of the Aviles-Giga functional. Analogous issues arise in the context of a closely related micromagnetics energy \cite{ser1,ser2, AKLR02, deott1, ser3} and in the study of large deviation
principles for some stochastic processes, where the limiting equations are
scalar conservation laws \cite{BBMN10}. Roughly speaking it is expected that vector fields $m\in\AI(\Omega)$ exhibit properties similar to those enjoyed by $BV$ functions and that entropy measures $\dv\Phi(m)$ are concentrated on a one-dimensional rectifiable set on which $m$ has left and right traces. The most progress to date in this direction is due to De Lellis and  Otto \cite{ottodel1}, who showed that the points of positive one-dimensional
density of entropy measures do form an $\mathcal{H}^1$-rectifiable set $J$.  However their result leaves open concentration of entropy measures on this set $J$. Indeed, a major conjecture raised in \cite{ottodel1} is the following
\begin{con}[De Lellis-Otto]
\label{doc}
For any $m\in \AI(\Omega)$, $\dv \Phi(m)$ is supported on an $\mathcal{H}^1$ $\sigma$-finite rectifiable set $J$ for all entropies $\Phi$ and 
\begin{equation*}
\dv \Phi(m)=\lt[\eta\cdot \lt(\Phi\lt(m^{+}\rt)- (\Phi\lt(m^{-}\rt)  \rt)\rt] \mathcal{H}^1_{\lfloor J},
\end{equation*}
where $\eta$ is the unit vector normal to $J$ and $m^{\pm}$ are the traces of $m$ on the two sides of $J$.
\end{con}
It is expected that such concentration of entropy measures on the $\mathcal{H}^1$-rectifiable set $J$, if resolved, will be a crucial step towards the full proof of the Aviles-Giga conjecture. Note that very recently Marconi has resolved the analogous versions of Conjecture \ref{doc} for Burgers equation \cite{elio2} and the micromagnetics functional \cite{elio} using a powerful Lagrangian representation method.

\subsection{On the threshold regularity for the Eikonal equation}

As described in the introduction of \cite{csz}, a recurring theme in modern non-linear PDE is to understand the threshold between rigidity and flexibility that often occurs at the limiting regularity required for weak solutions of PDE to satisfy additional \em derived  equations \rm by virtue of some form of the chain rule. In 
\cite{csz} the list of examples presented includes entropy solutions of hyperbolic conservation laws and incompressible Euler equations. The authors further note that it is an interesting future direction to understand the rigidity/flexibility threshold on Besov or Sobolev scale for non-linear PDE.

The notion of entropies for the Aviles-Giga functional and the Eikonal equation is closely connected to (and indeed was inspired by) entropies for hyperbolic conservation laws. For $|m|=1$, it is sufficient to define entropies on $\mathbb{S}^1$. It is easy to check that the restrictions to $\mathbb{S}^1$ of entropies defined in \cite{mul2} satisfy
\begin{equation}
\label{cpa1.7}
e^{it}\cdot \frac{d}{dt}\Phi(e^{it})=0. 
\end{equation}
Thus, following \cite{GL} we define the set of entropies to be
\begin{equation}
\label{cpa1.9}
ENT:=\lt\{ \Phi\in C^2\lt(\mathbb{S}^1;\mathbb{R}^2\rt): \Phi\text{ satisfies }(\ref{cpa1.7}) \rt\}.
\end{equation}
The derived equations for the Eikonal equation are therefore the equations $\dv\Phi(m)\equiv 0$ for all $\Phi\in ENT$. In \cite{DeI} (building on previous work \cite{ignat}) the authors showed that for $m$ satisfying the Eikonal equation \eqref{ageq3}, the Sobolev regularity $m\in W^{\frac{1}{3},3}_{\loc}$ is sufficient to allow something like the chain rule (or more accurately a substitute for the chain rule) to hold and hence conclude that all the derived equations hold. By the fundamental result of \cite{otto}, this is strong enough information to conclude that $m$ is rigid in the sense that it has only vortex singularities; see Conjecture \ref{CC2} below for the precise statement. On the other hand, in \cite[Theorem 2.6]{GL} it is shown that for weak solutions $m$ to the Eikonal equation \eqref{ageq3}, the Besov regularity $m\in B^{\frac{1}{3}}_{3,\infty,\loc}(\Omega)$ is equivalent to $\dv \Phi(m)$ being locally finite measures for all $\Phi\in ENT$, and thus there is no rigidity for $m$ under such regularity. These results leave the question of the critical regularity threshold between rigidity/flexibility for the Eikonal equation, and our first main result provides an initial step in this direction.

Our first theorem is an $L^p$ version of the main theorem in \cite{GL}. To state this result, we first make some definitions.  Recall that we denote by $\mathcal{M}(\Omega)$ the set of finite Radon measures on $\Omega$. For $\mu\in\mathcal{M}(\Omega)$, let $\|\mu\|$ denote its total variation measure. Next we recall \cite[Definition 2.5]{GL} with some extensions for our setting:
\begin{defin}
	\label{def:KIN}
	 We say that a vector field $m$ solving \eqref{ageq3} satisfies the \emph{kinetic formulation} if there exists a Radon measure $\sigma\in\mathcal{M}_{\loc}(\Omega\times \R/2\pi\mathbb{Z})$ such that
	 \begin{equation}
	 	\label{eqinta3}
	 	e^{is}\cdot \na_x \mathds{1}_{e^{is}\cdot m(x)>0}=\partial_s \sigma\qd\text{ in }\mathcal{D}'\lt(\Omega\times \mathbb{R}/ 2\pi\mathbb{Z}\rt). 
	 \end{equation}  
 	We call the measure $\sigma$ the \emph{kinetic measure}.
 	
 	Further, we say that $m$ satisfies the \emph{$L^p$ kinetic equation} for some $1\leq p<\infty$ if it satisfies \eqref{eqinta3} and there exists a family of measures $\lt\{\sigma_x\rt\}\subset \mathcal{M}(\mathbb{R}/ 2\pi \mathbb{Z})$ for $\mathcal{L}^2$-a.e. $x\in\Omega$ such that 
 	\begin{equation}
 		\label{eq30.2}
 		\la \sigma(x, s), f(s)\zeta(x) \ra=\int_{\Omega}\lt(\int_{\mathbb{R}/ 2\pi \mathbb{Z}}  f(s) d\sigma_x(s)\rt) \zeta(x) \, dx\qquad\forall f\in C^0\lt(\mathbb{R}/ 2\pi \mathbb{Z}\rt)  ,\:\zeta\in C^0_c(\Omega)
 	\end{equation}
 	and 
 	\begin{equation}
 		\label{eq:sigmaLp}
 		\nu(x):=\|\sigma_x\|_{\mathcal{M}(\mathbb{R}/ 2\pi \mathbb{Z})}\in L^p_{\loc}(\Omega).
 	\end{equation}
\end{defin}
%
When \eqref{eq30.2} and \eqref{eq:sigmaLp} hold, following the notation of \cite[Definition 2.27]{ambrosio}, we write $\sigma=\mathcal{L}^2\otimes \sigma_x$. Our first main result is the following
\begin{thm}
\label{C2}
Let $\Omega\subset\R^2$ be an open set and $m:\Omega\rightarrow \R^2$ satisfy \eqref{ageq3}. Then for any $1<p\leq \frac 43$, the following are equivalent:
\begin{enumerate}[(A)]
\item $m\in B^{\frac{1}{3}}_{3p,\infty,\loc}(\Omega)$;
\item  $\dv\Phi(m)\in L^p_{\loc}(\Omega)$ for all $\Phi \in ENT$;
\item $m$ satisfies the $L^{p}$ kinetic equation in the sense of Definition \ref{def:KIN}.
\end{enumerate}
\end{thm}
Note that for $p>\frac 43$, we are unable to establish the implication $(C)\Longrightarrow (A)$ in Theorem \ref{C2}. However, for $p\geq \frac 43$, knowing only two special entropy productions of $m$ (those of the Jin-Kohn entropies given in \eqref{ep101} and \eqref{ep102} in the following subsection) are in $L^p_{\loc}$ is sufficient to establish the $B^{\frac 13}_{4,\infty,\loc}$ regularity for $m$ (see Proposition \ref{LLD1}). Our interest in Theorem \ref{C2} is two-fold. Firstly, concerning Conjecture \ref{doc}, for $m\in\AI(\Omega)$, it is not even known if entropy measures are singular with respect to the Lebesgue measure and a proof of this fact would represent the first progress on this long standing problem.  The following conjecture represents an even more accessible goal in this program:
\begin{con}
	\label{CC2}
	Let $m:\Omega\rightarrow \R^2$ satisfy \eqref{ageq3}. Assume $\dv\Phi(m)\in L^p_{\loc}(\Omega)$ for some $p>1$ and all $\Phi \in ENT$,
	then  $\dv \Phi(m)\equiv 0$ for all $\Phi\in ENT$, and hence $m$ is rigid, i.e. $m$ is locally Lipschitz outside a locally finite set of points, and in any neighborhood containing only one singularity $m$ forms a vortex. 
\end{con}
Our Theorem \ref{C2} gives equivalent formulations for $\dv\Phi(m)\in L^p_{\loc}(\Omega)$ for all $\Phi \in ENT$, in particular in terms of the Besov regularity $m\in B^{\frac{1}{3}}_{3p,\infty,\loc}(\Omega)$. Under this Besov regularity, many powerful tools become available. In particular, in Theorem \ref{C1}, we establish explicit formulas for entropy productions and the kinetic measure $\sigma$ in terms of two special entropies. It is likely that one can extract very detailed information out of the explicit formulas, which will lead to rigidity of $m$ as stated in Conjecture \ref{CC2}. 

Secondly, it is known that the Eikonal equation enjoys no rigidity for $m\in B^{\frac{1}{3}}_{3,\infty,\loc}$. Thus a consequence of (a) Theorem \ref{C2} and (b) a proof of Conjecture \ref{CC2} 
would provide the threshold of rigidity/flexibility for the Eikonal equation at the Besov scale. Specifically, let $\mathcal{S}_m$ denote the points  of approximate discontinuity of $m$ (see  \cite[Definition 3.63]{ambrosio}), then  
\begin{equation*}
(a)\,\&\, (b)\Longrightarrow 
m\in \begin{cases}   B^{\frac{1}{3}}_{3,\infty,\loc}(\Omega)  & \mathcal{S}_m\text{ can be }\mathcal{H}^1\text{-rectifiable};\\
 B^{\frac{1}{3}}_{q,\infty,\loc}(\Omega), q>3   & \mathcal{S}_m\text{ can only be isolated points and }m\\ &\text{forms a vortex around each point in }\mathcal{S}_m.    \end{cases}
\end{equation*}

\subsection{Control of all entropies by two special ones}
Our next main results explore the connections between general entropies and two special entropies. Building from earlier work of Aviles and Giga \cite{avilesgig},  Jin and Kohn \cite{JK} introduced two fundamental entropies $\Sigma_1, \Sigma_2:\R^2\to\R^2$ given by
\begin{equation}
\label{ep101}
\Sigma_1(m)=\lt(m_2\lt(1-m_1^2-\frac{m_2^2}{3}\rt), m_1\lt(1-m_2^2-\frac{m_1^2}{3}\rt)\rt)
\end{equation}
and
\begin{equation}
\label{ep102}
 \Sigma_2(m)=\lt(-m_1\lt(1-\frac{2 m_1^2}{3}\rt), m_2\lt(1-\frac{2 m_2^2}{3}\rt)\rt).
\end{equation}
These two entropies restricted to $\mathbb{S}^1$ satisfy \eqref{cpa1.7}, and thus belong to the set $ENT$ defined in \eqref{cpa1.9}. They play fundamental roles in the $\Gamma$-limit conjecture for the Aviles-Giga functional. Indeed, the conjecture in \cite{ADM} is that the $\Gamma$-limit of the Aviles-Giga functional is (up to a constant) the total mass of the entropy measure
\begin{equation*}
	\mu = \left\Vert \begin{array}{c} \dv\Sigma_1(\nabla^\perp u)\\\dv\Sigma_2(\nabla^\perp u)\end{array}\right\Vert,
\end{equation*}
which is indeed controlled by the energy, and coincides with the cubic jump cost when $\nabla u\in BV$. A necessary condition for this $\Gamma$-limit conjecture to hold is that the two special entropy measures $\dv\Sigma_j(m)$ control all entropy measures, and our remaining results are in this spirit. 

To state our next result, for $m:\Omega\to\R^2$, $x\in\Omega$ and $0<\ep<\mathrm{dist}\{x,\partial\Omega\}$, we define the function
\begin{equation}
\label{eqssc1}
\PPI_m^{\ep}(x):=\ep^{-3}\int_{B_{\ep}(0)} \lt|D^z m(x)\rt|^3 dz, 
\end{equation}
where we denote $D^z m(x)=m(x+z)-m(x)$. Let us take a moment to explain the function $\PPI_m^\e$. This function arises naturally in the control of entropy productions of $m$ satisfying \eqref{ageq3}. Indeed, letting $m_\e$ be a regularization of $m$, direct computations show that the main contribution of $\dv\Phi(m_\e)$ comes from terms of the form $|\na m_\e|\lt(1-|m_\e|^2\rt)$ for any $\Phi\in ENT$ (see \eqref{eqllkk1}). A commutator argument further gives $|\na m_\e|\lt(1-|m_\e|^2\rt)\lesssim\PPI_m^\e$. On the other hand, if $m\in B^{\frac 13}_{3p,\infty}$ for some $p>1$, then from Lemma \ref{lemp6} we have $\norm{\PPI_m^{\ep}}_{L^{p}}\lesssim |m|_{B^{\frac{1}{3}}_{3p,\infty}}^{3}$. Thus, putting everything together, if $m$ is sufficiently regular (which is the case if $m\in B^{\frac 13}_{3p,\infty}$ for some $p>1$), one can pass to the limit as $\e\to 0$ to deduce
\begin{equation}
	\label{eq:P_m^e}
	\norm{\dv\Phi(m)}_{L^p}\lesssim \norm{\PPI_m}_{L^{p}}\lesssim |m|_{B^{\frac{1}{3}}_{3p,\infty}}^{3}
\end{equation}
for some $\PPI_m$ which is a weak limit of $\PPI_m^\e$. As such, the function $\PPI_m^\e$ plays a crucial role in controlling entropy productions of $m$ by its Besov norm, and we will make extensive use of it.

In \cite{GL}, Ghiraldin and Lamy constructed a special class of entropies $\{\Phi_f\}\subset ENT$ parameterized by $f\in C^{0}(\R/2\pi\mathbb{Z})$ (see \eqref{eqplm1}--\eqref{eqplm1.5} in Section \ref{profthm3A}). They proved the existence of the
kinetic measure $\sigma$ (recalling (\ref{eqinta3})) for any $m$ satisfying \eqref{ageq3} and $\dv \Phi_f(m)\in \mathcal{M}_{\loc}(\Omega)$ for all $f\in C^0(\mathbb R/ 2\pi \mathbb{Z})$, i.e.\ the family of entropies $\lt\{\Phi_f\rt\}$ is rich enough to generate the kinetic equation \eqref{eqinta3}. Our next result gives explicit formulas for $\dv\Phi_f(m)$ and the kinetic measure $\sigma$ in terms of $\dv\Sigma_j(m)$, $j=1, 2$, under suitable assumptions. 

\begin{thm}
\label{C1}
Let $\Omega\subset\R^2$ be an open set and $m:\Omega\rightarrow \R^2$ satisfy \eqref{ageq3}, and let $\PPI_m^{\ep}$ be defined by \eqref{eqssc1}. Assume for some sequence $\ep_k\rightarrow 0$ and $1\leq p<\infty$, there exists $\PPI_m\in L^p_{\loc}(\Omega)$ such that
\begin{equation}
\label{cpeq1.23}
\PPI_m^{\ep_k}\rightharpoonup \PPI_m\qd\text{ in }L^p_{\loc}(\Omega).
\end{equation}
Then  $\dv \Phi_f(m)\in L^{p}_{\loc}(\Omega)$ for all $f\in C^0\lt(\mathbb{R}/ 2\pi \mathbb{Z} \rt)$, where $\Phi_f\in ENT$ is defined in \eqref{eqplm1}. Further, it explicitly holds
\begin{align}
\label{cpeq2}
\dv \Phi_f(m)=\frac{1}{2}\lt(f\lt(\theta+\frac{\pi}{2}\rt)+f\lt(\theta-\frac{\pi}{2}\rt)-2\la f,1\ra \rt) e^{i 2\theta}\cdot \dv \Sigma(m)\qd\text{ a.e. in }\Omega,
\end{align}
where $\theta: \Omega\to[0,2\pi)$ satisfies $m(x)=e^{i\theta(x)}$ for a.e. $x\in\Omega$, $\Sigma(m)=\lt(\Sigma_1(m),\Sigma_2(m)\rt)$ and $\langle f, g\rangle = \frac{1}{2\pi}\int_0^{2\pi}f g$. Equivalently, \eqref{cpeq2} can be formulated as the disintegration of the kinetic measure $\sigma=\mathcal{L}^2\otimes\sigma_x$, where 
\begin{equation}
\label{abeqa1.5}
\sigma_x=\frac{1}{2}\lt(\delta_{\theta(x)+\frac{\pi}{2}}+\delta_{\theta(x)-\frac{\pi}{2}}-\frac{1}{\pi}\mathcal{L}^1\rt)e^{i2\theta(x)}\cdot\dv\Sigma(m)(x) \qd\text{ for a.e. } x\in\Omega.
\end{equation} 
\end{thm}

\begin{rem}
	By Lemmas \ref{lemp6}, \ref{p:controlent} and Theorem \ref{C2}, the assumption \eqref{cpeq1.23} is equivalent to $\dv\Phi(m)\in L^p_{\loc}(\Omega)$ for all $\Phi \in ENT$ and equivalent to $m\in B^{\frac 13}_{3p,\infty,\loc}(\Omega)$ for $1<p\leq \frac 43$. When $p>\frac 43$, Lemmas \ref{lemp6} and \ref{p:controlent} imply that 
	\begin{equation*}
	m\in B^{\frac 13}_{3p,\infty,\loc}(\Omega)\Longrightarrow \eqref{cpeq1.23}\Longrightarrow \dv\Phi(m)\in L^p_{\loc}(\Omega) \text{ for all } \Phi \in ENT,
	\end{equation*}
	however we are unable to establish equivalence. Nevertheless, the conclusions of Theorem \ref{C1} still hold under the assumption $\dv\Phi(m)\in L^p_{\loc}(\Omega)$ for all $\Phi \in ENT$ when $p>\frac 43$; see Theorem \ref{C0} below. When $p=1$, we only have the implication $\eqref{cpeq1.23}\Longrightarrow \dv\Phi(m)\in L^1_{\loc}(\Omega)$ for all $ \Phi \in ENT$ by Lemma \ref{p:controlent}. Again we are unable to establish equivalence in this case, nor do we have a proof of Theorem \ref{C1} with \eqref{cpeq1.23} replaced by the assumption $\dv\Phi(m)\in L^1_{\loc}(\Omega)$ for all $\Phi \in ENT$. 
\vspace{.05in}

Finally note that after writing up our factorization result, we became aware of an upcoming similar but stronger result being established by Elio Marconi \cite{marconi20} by use of the theory of ``Lagrangian representation" \cite{biamar, biamar2, marconi19,elio2,elio}.
\end{rem}

As a consequence of the explicit formula \eqref{cpeq2}, the entropy production $\dv\Phi_f(m)$ is controlled by $\dv\Sigma(m)$ in a very precise manner for the class of entropies $\{\Phi_f\}$ under the assumption \eqref{cpeq1.23}.  This formula could potentially have wider applications in the study of the more ``regular'' part of entropy measures for general $m\in\AI(\Omega)$, i.e. the part concentrated outside the $\mathcal{H}^1$-rectifiable set. Note that by similar methods to the proof of Theorem \ref{C2}, it is possible to show that $\lt|e^{i 2\theta}\cdot \dv \Sigma\lt(m\rt)\rt|\geq \pi^{-5}\PPI_m$ a.e. in $\Omega$. As a somewhat direct application of the above Theorems \ref{C2} and \ref{C1}, we have
 
\begin{thm}
\label{C0}
Let $\Omega\subset\R^2$ be an open set and $m:\Omega\rightarrow \R^2$ satisfy \eqref{ageq3}. Assume either 
\begin{equation}
\label{eqhyp1}
\dv\Phi(m)\in L^p_{\loc}(\Omega)\qd\text{ for some }p>1\text{ and all }\Phi\in ENT
\end{equation}
or 
\begin{equation}
\label{eqhyp2}
 \dv \Sigma_j(m)\in L^p_{\loc}(\Omega)\qd\text{ for some }p\geq\frac{4}{3}\text{ and }j=1, 2.
\end{equation}
Then $\dv \Phi_f(m)\in L^{p}_{\loc}(\Omega)$ for all $f\in C^0\lt(\mathbb{R}/ 2\pi \mathbb{Z} \rt)$, and further \eqref{cpeq2} and \eqref{abeqa1.5} hold true. 
\end{thm}
The statements of Theorems \ref{C1} and \ref{C0} are for the class of entropies $\{\Phi_f\}$ essentially because this is a wide enough class to generate the kinetic measure and the factorization has a neater form. In Proposition \ref{L17} we also establish the corresponding factorization formula for sufficiently regular entropies $\Phi\in ENT$. Note that the conclusions of Theorem \ref{C0} under the assumption \eqref{eqhyp2} is a result to control \emph{all} entropies by the two special entropies $\Sigma_j$ in the strong sense that we only need $\dv \Sigma_j(m)\in L^p_{\loc}(\Omega)$ to conclude that $\dv \Phi_f(m)\in L^{p}_{\loc}(\Omega)$ for all $\Phi_f$ with the additional explicit formula \eqref{cpeq2}. This is a strong extension of the main result in \cite{LP}, which can be formulated as an immediate corollary:
\begin{cor}[\cite{LP}]
\label{C01}
Let $\Omega\subset\R^2$ be an open set and $m:\Omega\rightarrow \R^2$ satisfy \eqref{ageq3}. Assume $\dv \Sigma_j(m)= 0$ in $\mathcal{D}'(\Omega)$ for $j=1,2$, then $\dv\Phi_f(m)=0$ in $\mathcal{D}'(\Omega)$ for all $f\in C^0\lt(\mathbb{R}/ 2\pi \mathbb{Z} \rt)$  and the kinetic measure $\sigma$ vanishes. 
\end{cor}
When $\sigma$ vanishes, the right-hand side of the kinetic equation \eqref{eqinta3} also vanishes. It then follows from the main result of \cite{otto} that $m$ is rigid in the sense of Conjecture \ref{CC2}, and this recovers \cite[Theorem 3]{LP}. Corollary \ref{C01} and Theorem \ref{C0} (in particular the formula \eqref{cpeq2}) give strong evidence towards the $\Gamma$-limit for the Aviles-Giga functional conjectured in \cite{ADM}. 
\nl

\noindent\bf  Acknowledgments. \rm We would like to extend our deep thanks to Xavier Lamy for innumerable very helpful conversations and specifically for the proof of Proposition  \ref{BimplyC} and Lemmas \ref{lem:sigmaLp} and \ref{p:sigmaLp}, and for most of the ideas of the proof of Proposition \ref{LLD1} and Lemmas \ref{LB16} and
\ref{LLAA8}. The proof of Proposition \ref{BimplyC} helps strengthen an earlier version of our Theorem \ref{C2}. Also note that Proposition \ref{p:harmext} (including Lemmas \ref{l:Phifk} and \ref{l:extPhifk}) was originally in the first posted version of \cite{llp} and as such is joint work of the three of us. A. L. gratefully acknowledges the support of the Simons foundation, collaboration grant \#426900.

%
%

\section{Proof sketch}

\subsection{Extension of entropies to $\overline B_1$ and commutator argument} 
\label{sec:keyideas}

A fundamental idea that we will use throughout is to try and understand $\dv \Phi(m)$ by understanding $\lim_{\ep\rightarrow 0} \dv \Phi\lt(m_{\ep}\rt)$ for a regularization $m_\ep$ of $m$. By definition, entropies are designed to provide additional derived equations   via the chain rule and as such applying them to $m_{\ep}$ is a natural idea. However for $\Phi\in ENT$, it is necessary to extend the definition to $\overline B_1$ in order to act them on $m_\e$. One natural way to do this is to extend $\Phi$ through a \emph{radial extension} (see \eqref{eqprin11}). Another way is through the original definition of entropies in \cite{mul2}. Here entropies are defined in $\R^2$ and obtained from smooth functions $\varphi$ via the formula
\begin{equation*}
	\Phi^{\bp{\varphi}}(z)=\varphi(z)z +((iz)\cdot\nabla\varphi(z))iz\qquad\forall z\in\R^2.
\end{equation*}
The underlying principle is, through different extensions, we extract different information from the limit of $\dv \Phi\lt(m_{\ep}\rt)$, and when pieced together, the information has strong consequences for the structure of $\dv \Phi\lt(m\rt)$.

In order to pass to the limit in $\dv \Phi\lt(m_{\ep}\rt)$, an important trick is the commutator argument of \cite{titi} that was first introduced to our setting in \cite{DeI}, and used later in \cite{LP,GL,llp} with extensions. As explained above \eqref{eq:P_m^e}, this commutator argument allows to control $\dv \Phi\lt(m_{\ep}\rt)$ and pass to the limit as $\ep\to 0$ to establish an estimate like \eqref{eq:P_m^e} under sufficient regularity of $m$. This is a crucial idea in many of our arguments.

\subsection{Ideas in the proof of Theorem \ref{C2}} 
The proof of Theorem \ref{C2} goes in the loop $(A)\Longrightarrow (B)$, $(B)\Longrightarrow (C)$ and $(C)\Longrightarrow (A)$ and uses many of the ideas in \cite{GL}. The first two steps use rather straightforward ideas, and the step $(C)\Longrightarrow (A)$ is the most involved requiring significant refinements of the estimates used in \cite{GL}. For this last step, the starting point is to create a function 
$\Delta_{\alpha}$ (as opposed to the $\Delta$ used in \cite{GL}) that provides the right coercivity estimate towards the goal of $B^{\frac{1}{3}}_{3p,\infty}$ regularity; see Lemma \ref{LB16}. On the other hand, as observed in \cite{GL}, (a regularized) $\Delta_{\alpha}$ enjoys a nice identity due to the kinetic equation \eqref{eqinta3} (see \eqref{partial_h}). Then using more delicate arguments based on many ideas from \cite{GL} (what the authors called the \em interaction estimate \em due to Varadhan (see \cite[Chapter 22]{tartar}) as used in \cite{golse,gold}), we establish in Lemmas \ref{lem:I_alpha} and \ref{lem:A_alpha} appropriate upper bound estimate for $\Delta_\alpha$, which, coupled with the coercivity established in Lemma \ref{LB16}, gives the desired $B^{\frac{1}{3}}_{3p,\infty}$ regularity. 

\subsection{Ideas in the proof of Theorem \ref{C1}} 
Our starting point is to use computations from \cite{llp} to establish in Lemma \ref{lemp5} the formula
	\begin{equation}\label{ep1121}
	\dv\Phi(m)=A^{\Phi}_1(m) e^{i2\theta}\cdot\dv\Sigma(m)+ A^{\Phi}_2(m) i e^{i2\theta}\cdot\dv\Sigma(m) \qd\text{ a.e. in }\Omega
\end{equation}
for general entropies $\Phi$. To this end, we use ideas presented in Subsection \ref{sec:keyideas}, namely, using regularization of $m$ and appropriate extension of $\Phi$ to $\overline B_1$ to compute $\dv\Phi(m_\e)$, and passing to the limit through a commutator argument. Such ideas are not hard to implement in the case $p>1$. In the case $p=1$, it is a bit more involved and the technical aspect is handled in  Lemma \ref{lembite}. 

The rest of the game involves showing $i e^{i2\theta}\cdot\dv\Sigma(m)=0$ using a similar contradiction argument as in \cite{llp}. Specifically, assuming $i e^{i2\theta}\cdot\dv\Sigma(m)\ne 0$ at a generic point $x$, then estimates established for the coefficients $A^{\Phi}_j$ in Lemma \ref{fmlem} plugged into \eqref{ep1121} would force $A^{\Phi}_2$ to satisfy a bound of the form $\abs{A_2^{\Phi}(m(x))}\lesssim\norm{\Phi_{\lfloor\mathbb S^1}}_{C^2(\mathbb S^1)}$ for all entropies $\Phi$. This estimate cannot hold due to the fact that the Hilbert transform (or conjugate function operator) on $\mathbb S^1$ is not bounded from $C^0(\mathbb S^1)$ to $L^\infty(\mathbb S^1)$. This contradiction shows $i e^{i2\theta}\cdot\dv\Sigma(m)=0$, and the formula \eqref{ep1121} reduces to more general version of \eqref{cpeq2} (see Proposition \ref{L17} for the details). Finally, \eqref{cpeq2} is obtained by recognizing the class of entropies $\{\Phi_f\}$ as harmonic entropies (first introduced in \cite{LP}) $\Phi\in ENT$, and this is the content of Subsection \ref{sec:harmext}.

\section{Preliminaries}

In this section we put together some notations and results that will be used repeatedly in later sections. Throughout this paper, we use the notation $A\lesssim B$ to indicate $A\leq c B$ for some constant $c$ independent of the underlying domain or functions. Let $\rho\in C^{\infty}_c(\R^2)$ be the standard convolution kernel with $\int_{\R^2}\rho=1$, and let $\rho_{\ep}(z)=\ep^{-2}\rho\lt(\frac{z}{\ep}\rt)$. Given a function $f$ we denote $f_{\ep}:=f*\rho_{\ep}$.  

Given a bounded Lipschitz domain $\Omega\subset\R^n$ and $f:\Omega\rightarrow \R$, $z\in\R^n$, we define 
\begin{equation*}
	D^zf(x):=\begin{cases}
		f(x+z)-f(x) &\text{ if } x,x+z\in\Omega;\\
		0 &\text{ otherwise}.
	\end{cases}
\end{equation*}
For any $s\in (0,1)$, $p,q\in [1,\infty]$, we set
\begin{equation*}
	\abs{f}_{B^s_{p,q}(\Omega)}=\left\Vert t^{-s}\sup_{\abs{h}\leq t}\norm{D^h f}_{L^p(\Omega)} \right\Vert_{L^q(\frac{dt}{t})},
\end{equation*}
and  the Besov space $B^s_{p,q}(\Omega)$ is the space of functions $f\colon\Omega\to\R$ such that
\begin{equation*}
	\norm{f}_{L^p(\Omega)} + \abs{f}_{B^{s}_{p,q}(\Omega)} <\infty.
\end{equation*}


%
%

\begin{lem}
\label{lemp6}
Given $m\in B^{\frac{1}{3}}_{3p,\infty,\loc}(\Omega)$ for some $1<p<\infty$, let $\PPI_m^{\ep}$ be given in \eqref{eqssc1}. Then for any $U\subset\subset\Omega$ and $\ep>0$ sufficiently small, we have
\begin{equation*}
\norm{\PPI_m^{\ep}}_{L^{p}(U)}\leq  |m|_{B^{\frac{1}{3}}_{3p,\infty}(U)}^{3}.
\end{equation*}
In particular $\{\PPI_m^\ep\}$ forms a bounded sequence in $L^{p}_{\loc}(\Omega)$.
\end{lem}

Lemma \ref{lemp6} is a trivial generalization of \cite[Lemma 8]{llp}. Its proof follows exactly the same lines by replacing $4/3$ with $p$ in the proof of the latter, and is thus skipped. The following lemma is a slight extension of \cite[Proposition 10]{llp}.

%
%
%
%

\begin{lem}\label{p:controlent} 
Let $m:\Omega\to\R^2$ satisfy \eqref{ageq3}. Assume \eqref{cpeq1.23} holds true for some sequence $\ep_k\rightarrow 0$ and $1\leq p<\infty$.   Then for any $\Phi\in  ENT$ (given in \eqref{cpa1.9}), we have $\dv\Phi(m)\in L^p_{\loc}(\Omega)$ and
\begin{equation}
\label{eq621.2}
\lt|\dv \Phi(m)(x)\rt|\lesssim  \|\Phi\|_{C^2\lt(\mathbb{S}^1\rt)}  \PPI_m(x) \qquad\text{ for a.e. }x\in \Omega.
\end{equation}
\end{lem} 
\begin{proof}
Since the proof is very similar to the proof of \cite[Proposition 10]{llp}, we only sketch it somewhat briefly. Given $\Phi\in ENT$, we introduce the extension $\wt{\Phi}:\R^2\to\R^2$ of $\Phi$ given by
\begin{equation}
	\label{eqprin11}
	\wt{\Phi}(z):=\eta\lt(\lt|z\rt|\rt) \Phi\lt(\frac{z}{\lt|z\rt|}\rt),
\end{equation}
where $\eta:[0,\infty)\to\R$ is a smooth cut-off function with $\eta=0$ in $[0,\frac{1}{2}]\cup[2,\infty)$ and $\eta(1)=1$. For a regularization $m_{\ep}$ of $m$, as in the proof of \cite[Proposition 10]{llp} (originally from Steps 2 and 3 of the proof of \cite[Proposition~3]{DeI}) we have
\begin{align}
\label{eqllkk1}
\dv \wt{\Phi}\lt(m_{\ep}\rt)&=\Psi\lt(m_{\ep}\rt)\cdot \na \lt(1-\lt| m_{\ep}\rt|^2\rt)\nn\\
&=\na\cdot\lt(\Psi(m_{\e})(1-|m_{\e}|^2)\rt)-\na\cdot\Psi(m_{\e})\lt(1-|m_{\e}|^2\rt),
\end{align} 
where $\Psi(z)=\frac{-D\wt{\Phi}(z) z+\gamma(z)z}{2 \lt|z\rt|^2}$ and $\gamma(z)=\frac{z^{\perp}\cdot D\wt{\Phi}(z)z^{\perp}}{|z|^2}$. 

Given any open set $U\subset\subset\Omega$, for any $\zeta\in C_c^{\infty}(U)$, an integration by parts using \eqref{eqllkk1} and the estimate $\|D\Psi\|_{L^{\infty}(\mathbb{R}^2)}\lesssim \|\Phi\|_{C^2\lt(\mathbb{S}^1\rt)}$ (see equation (35) in \cite{llp}) yield (with $\ep$ replaced by $\ep_k$)
\begin{align}
\label{eqvbnm1}
\lt|\int_{U} \wt{\Phi}(m_{\ep_k})\cdot \na \zeta dx\rt|&\lesssim
\|\Phi\|_{C^2\lt(\mathbb{S}^1\rt)}\int_{U} \lt|\na m_{\ep_k}\rt|\lt| 1-\lt|m_{\ep_k}\rt|^2\rt| \lt|\zeta\rt|dx\nn\\
&\qd\,+\lt|\int _{U}\lt(1-\lt| m_{\ep_k}\rt|^2\rt)\Psi\lt(m_{\ep_k}\rt)\cdot \na \zeta \, dx\rt|.
\end{align}
It is not hard to see that the left-hand side of \eqref{eqvbnm1} converges to $\lt|\int_{U} \Phi(m)\cdot \na \zeta dx\rt|$, and, upon extraction of a subsequence (not relabeled),  the second term on the right-hand side tends to zero as $k\to\infty$ by the Dominated Convergence Theorem. 
And by \cite[Lemma 9]{llp} we have that $ \lt|\na m_{\ep}\rt|\lt| 1-\lt|m_{\ep}\rt|^2\rt|\lesssim \PPI_m^{\ep}$. Passing to the limit as $k\to\infty$ in \eqref{eqvbnm1} and using \eqref{cpeq1.23}, we obtain
\begin{equation}
\label{eqokp1}
\lt|\int_U \Phi\lt(m\rt)\cdot \na\zeta\, dx \rt|\lesssim \|\Phi\|_{C^2\lt(\mathbb{S}^1\rt)}\int_U \PPI_m \lt|\zeta\rt| dx\leq\|\Phi\|_{C^2\lt(\mathbb{S}^1\rt)}\|\PPI_m\|_{L^p(U)}\|\zeta\|_{L^{p'}(U)}
\end{equation}
for all $\zeta\in C^{\infty}_c(U)$. 

When $1<p<\infty$, it follows from \eqref{eqokp1} and standard arguments that $\dv\Phi(m)\in L^p_{\loc}(\Omega)$ and the pointwise estimate \eqref{eq621.2} holds for all Lebesgue points of $\dv\Phi(m)$ and $\PPI_m$; see the end of the proof of \cite[Proposition~6]{llp} for the details. When $p=1$, one can extend the estimate \eqref{eqokp1} to all $\zeta\in C^0_c(U)$. It follows that $\dv\Phi(m)\in \mathcal{M}_{\loc}\lt(\Omega\rt)$ and, using \cite[Proposition 1.47]{ambrosio}, we infer
\begin{equation}
	\label{eq:controlENT}
	\|\dv\Phi(m)\|_{\mathcal{M}(U)}=\sup_{\zeta\in C^0_c(U),\|\zeta\|_{C^0}\leq 1}\lt|\langle \dv\Phi(m),\zeta\rangle\rt|\lesssim \|\Phi\|_{C^2\lt(\mathbb{S}^1\rt)}\|\PPI_m\|_{L^1(U)}
\end{equation}
for all open $U\subset\subset\Omega$. It is then standard to deduce from \eqref{eq:controlENT} that $\|\dv\Phi(m)\|$ is absolutely continuous with respect to the Lebesgue measure with $L^1_{\loc}$ density, and the same holds for $\dv\Phi(m)$. Finally,  for all Lebesgue points $x$ of $|\dv\Phi(m)|$ and $\PPI_m$, we deduce from \eqref{eq:controlENT} that
\begin{align*}
	|\dv\Phi(m)(x)|&=\lim_{r\rightarrow 0}\Xint{-}_{B_r(x)} |\dv\Phi(m)| \, dz\\
	&\lesssim \|\Phi\|_{C^2\lt(\mathbb{S}^1\rt)}\lim_{r\rightarrow 0}\Xint{-}_{B_r(x)} \PPI_m\, dz=\|\Phi\|_{C^2\lt(\mathbb{S}^1\rt)}\PPI_m(x),
\end{align*}
and this establishes \eqref{eq621.2} for $p=1$.
\end{proof}

\begin{rem}
	\label{R13}
	Note that Lemma \ref{p:controlent} in particular gives 
	\begin{equation*}
	\dv\Sigma_j(m)\in L^p_{\loc}(\Omega)\qd\text{ for }j=1,2, 
	\end{equation*}
	where $\Sigma_j$ are given in \eqref{ep101}, \eqref{ep102}, since $\Sigma_j$ agrees (up to a factor) with $\Phi_{f}$ on $\mathbb{S}^1$ for $f=\cos(2t)$ or $f=\sin(2t)$; see \cite[Remark 3.2]{GL}.
\end{rem}
%
%

\section{Proof of Theorem \ref{C2}}
\label{profthm3A}

In this section we prove Theorem \ref{C2} by establishing the loop of implications $(A)\Longrightarrow (B)$, $(B)\Longrightarrow (C)$ and $(C)\Longrightarrow (A)$ in the following three subsections. This structure of proof is parallel to that of the main theorem in \cite{GL}. The most involved step is the last step $(C)\Longrightarrow (A)$, which requires significant refinements of the estimates in the counterpart  of the proof of \cite[Theorem 2.6]{GL}. Note that the proofs of $(A)\Longrightarrow (B)$ and $(B)\Longrightarrow (C)$ work for all $1<p<\infty$, and the assumption $p\leq \frac 43$ in Theorem \ref{C2} is only needed in the proof of $(C)\Longrightarrow (A)$.

\subsection{Proof of $(A)\Longrightarrow (B)$ in Theorem \ref{C2}}

\begin{prop}
	\label{LL16}
	Let $m:\Omega\to\R^2$ satisfy \eqref{ageq3}. Assume $m\in B^{\frac{1}{3}}_{3p,\infty,\loc}(\Omega)$ for some $1<p<\infty$, then $\dv\Phi(m)\in L^p_{\loc}(\Omega)$ for all $\Phi \in ENT$. 
\end{prop}
\begin{proof} From Lemma \ref{lemp6} we know $\{\PPI^{\e}_m\}$ forms a bounded sequence in $L^p_{\loc}(\Omega)$. As $1<p<\infty$, upon extraction of a subsequence, we have $\PPI^{\e_k}_m\rightharpoonup \PPI_m$ in $L^p_{\loc}(\Omega)$ for some $\PPI_m$. This allows us to apply Lemma \ref{p:controlent} to deduce for all $\Phi\in  ENT$ that $\dv\Phi(m)\in L^p_{\loc}(\Omega)$.
\end{proof}


\subsection{Proof of $(B)\Longrightarrow (C)$ in Theorem \ref{C2}}

\begin{prop}[Lamy]
	\label{BimplyC}
	Let $m:\Omega\to\R^2$ satisfy \eqref{ageq3}. Assume $\dv\Phi(m)\in L^p_{\loc}(\Omega)$ for some $1<p<\infty$ and all $\Phi \in ENT$, then $m$ satisfies the $L^p$ kinetic equation in the sense of Definition \ref{def:KIN}. 
\end{prop}

The proof of the above proposition relies crucially on the following Lemma \ref{lem:sigmaLp}, which is an $L^p$ version of \cite[Lemma 3.4]{GL}. Note that the proof of \cite[Lemma 3.4]{GL} is a bit unclear in that the estimates for tensor products established there are not sufficient to apply the Riesz representation theorem. Nevertheless as shown to us by Lamy, the generalized Riesz representation theorem stated in \cite{diest} overcomes these obstacles. In the following we include this complete proof adapted to our $L^p$ setting. In Appendix \ref{a:lem3.4}, we also include the complete proof of \cite[Lemma 3.4]{GL} since this result is needed in our Theorem \ref{C1} and the complete proof has not appeared elsewhere. We warmly thank Lamy for these proofs presented.

To begin with, we recall a class of entropies $\{\Phi_f\}\subset ENT$ constructed in \cite{GL}. Given  $f\in C^0(\mathbb R/ 2\pi \mathbb{Z})$, identifying  $\mathbb{S}^1\cong\mathbb R/2\pi\mathbb Z$ and $\mathbb R^2\cong\mathbb C$, we define  
\begin{equation}
	\label{eqplm1}
	\Phi_f(e^{it}):=-i\varphi_f(t-\pi/2)+i\varphi_f(t+\pi/2),
\end{equation}
where 
\begin{equation}
	\label{eq:varphi_f}
	\varphi_f(t):=\int_0^t \psi_f(s)ie^{is}ds
\end{equation} 
and 
\begin{equation}
	\label{eqplm1.5}
	\psi_f(t):=\int_0^t \left[f(s)-\langle f,1\rangle -2 \langle f, \cos \rangle \cos(s) -2 \langle f,\sin\rangle \sin(s)\right]ds,
\end{equation} 
and $\la \cdot, \cdot \ra$  denotes the inner product on $L^2\lt(\mathbb{R}/ 2\pi \mathbb{Z}\rt)$ defined by $\la f,g\ra:=\frac{1}{2\pi}\int_0^{2\pi} f(t)g(t)\, dt$.

\begin{lem}[Lamy]
	\label{lem:sigmaLp}
	Let $m:\Omega\to\R^2$ satisfy \eqref{ageq3} and $1<p<\infty$. Assume $\dv\Phi_f(m)\in L^p(\Omega)$ for all $f\in C^0(\mathbb S^1)$, where $\Phi_f$ is defined in \eqref{eqplm1}. Then there exists $\sigma\in L^p_w(\Omega;\mathcal M(\mathbb{R}/ 2\pi\mathbb{Z}))$, where the subscript $w$ stands for weak measurability, such that
	\begin{equation}
		\label{eqgl1}
		\langle \dv\Phi_f(m),\zeta\rangle =\langle\sigma(x, s),f(s)\zeta(x)\rangle\qquad\forall f\in C^0(\R/2\pi\mathbb{Z}),\:\zeta\in L^{p'}(\Omega),
	\end{equation}
	where $\langle\cdot,\cdot\rangle$ denotes duality in the appropriate sense. Consequently, $\sigma$ satisfies \eqref{eq30.2} and \eqref{eq:sigmaLp}.
\end{lem}

\begin{proof}
	 Given Banach spaces $X$ and $Y$, let $\mathcal{L}(X;Y)$ denote the space of bounded linear operators from $X$ to $Y$. Given $\zeta\in C^1_{c}(\Omega)$, define the linear operator $T_{\zeta}:C^0(\mathbb{R}/ 2\pi\mathbb{Z})\to\R$ to be
	\begin{equation*}
		T_{\zeta}(f):=\langle \dv\Phi_f(m),\zeta \rangle.
	\end{equation*}	
	The estimate
	\begin{equation*}
		|T_{\zeta}(f)|\leq \lVert\Phi_f\rVert_{L^{\infty}} \lVert\na\zeta\rVert_{L^{\infty}}\leq C\lVert\na\zeta\rVert_{L^{\infty}}\lVert f\rVert_{L^{\infty}}
	\end{equation*}
	shows that $T_{\zeta}\in \mathcal L(C^0(\mathbb{R}/ 2\pi\mathbb{Z});\R)$ for all $\zeta\in C^1_{c}(\Omega)$. For the convenience of the proof of Lemma \ref{p:sigmaLp} in Appendix \ref{a:lem3.4}, let $\mathcal{Z}=L^{p'}(\Omega)$, so that $\mathcal{Z}^*=L^p(\Omega)$, and define
	\begin{align*}
		\mathbb{W}_{\mathcal{Z}}:=\left\lbrace T_{\zeta}: \zeta\in C^1_{c}(\Omega), \norm{\zeta}_{\mathcal{Z}}\leq 1\right\rbrace.
	\end{align*}
	By hypothesis, for each fixed 
	$f\in C^0(\mathbb{R}/ 2\pi\mathbb{Z} )$ we have
	\begin{equation*}
	\sup_{\zeta\in C_c^1(\Omega), \|\zeta\|_{\mathcal{Z}}\leq 1 } \langle \dv\Phi_f(m),\zeta \rangle<\infty.
	\end{equation*}
	So by Banach-Steinhaus' uniform boundedness principle applied to $\mathbb{W}_{\mathcal{Z}}$ we have 
	\begin{align}
		\label{BanachSteinhaus}
		\langle \dv\Phi_f(m),\zeta \rangle \leq C \norm{f}_{C^0(\mathbb{R}/ 2\pi \mathbb{Z})}\norm{\zeta}_{\mathcal{Z}}\qquad\forall f\in C^0(\mathbb{R}/ 2\pi\mathbb{Z} ),\:\zeta\in C^1_c(\Omega).
	\end{align}
	By extension, the above estimate \eqref{BanachSteinhaus} also holds for all $\zeta\in \mathcal{Z}$, and therefore the map
	\begin{align}
	\label{fineqa1}
		T\colon C^0(\mathbb{R}/ 2\pi\mathbb{Z} )\to \mathcal{Z}^*,\quad f\mapsto \dv\Phi_f(m)
	\end{align}
	is a bounded linear operator.
	
	Since $1<p<\infty$, the space $L^p(\Omega)$ is reflexive. Thanks to \cite[Theorem VI.2.1]{diest}, there exists an $\mathcal{Z}^*$-valued Borel measure $G$ on $\mathbb{R}/ 2\pi\mathbb{Z}$ such that
	\begin{equation}
		\label{eqokl2}
		Tf = \int_{\mathbb{R}/ 2\pi\mathbb{Z}}f\,dG.
	\end{equation}
	We use this representation to define $\langle \sigma,\psi\rangle$ for a convenient class of functions $\psi(x,s)$. First consider $\psi$ a finite linear combination of the form
	\begin{align*}
		\psi(x,s)=\sum_j \one_{E_j}(s)\zeta_j(x),\qquad E_j\subset \mathbb{R}/ 2\pi\mathbb{Z} \text{ Borelian, }\zeta_j\in \mathcal{Z},
	\end{align*}
	and set
	\begin{align}
		\label{eqokl1}
		\langle\sigma,\psi\rangle :=\sum_j \langle G(E_j),\zeta_j\rangle.
	\end{align}
	Assuming without loss of generality that the $E_j$'s are disjoint and non-negligible, for such $\psi$ we have
	\begin{align}
		\label{eqp2}
		\langle\sigma,\psi\rangle & \leq  \sum_j\norm{G}(E_j)\norm{\zeta_j}_{\mathcal{Z}}\nn\\
		& \leq \norm{G}(\mathbb{R}/ 2\pi\mathbb{Z} )\max_j \norm{\zeta_j}_{\mathcal{Z}} =\norm{T}_{\mathcal L(C^0(\mathbb{R}/ 2\pi\mathbb{Z} );\mathcal{Z}^*)}\norm{\psi}_{L^{\infty}(\mathbb{R}/ 2\pi\mathbb{Z};\mathcal{Z})}.
	\end{align}
	As a consequence the linear form $\sigma$ admits a unique linear continuous extension (still denoted $\sigma$) to the space of all limits of $\psi$'s of the above form in the $L^{\infty}(\mathbb{R}/ 2\pi\mathbb{Z};\mathcal{Z})$ norm. Hence $\sigma$ can in particular be considered as a continuous linear form on $C^0(\mathbb{R}/ 2\pi\mathbb{Z};\mathcal{Z})$. Moreover, since for $f=\one_{E}$ and $\zeta\in \mathcal{Z}$, the definition (\ref{eqokl1}) gives
	\begin{align*}
		\langle \sigma(x,s),f(s)\zeta(x)\rangle = \langle \int f \, dG, \zeta\rangle,
	\end{align*}
	this formula must also be valid for $f\in C^0(\mathbb{R}/ 2\pi\mathbb{Z} )$ and we deduce that
	\begin{align*}
		\langle\sigma(x,s),f(s)\zeta(x)\rangle &= \langle \int f \, dG, \zeta\rangle
		\overset{(\ref{eqokl2})}{=}\langle Tf,\zeta\rangle\overset{(\ref{fineqa1})}{=}\la \dv \Phi_f(m), \zeta\ra   \qquad\forall f\in C^0(\mathbb{R}/ 2\pi\mathbb{Z}),\:\zeta\in \mathcal{Z},
	\end{align*}
	and this establishes \eqref{eqgl1}.

Since $C^0(\mathbb{R}/ 2\pi\mathbb{Z};L^{p'}(\Omega))\supset L^{p'}(\Omega;C^0(\mathbb{R}/ 2\pi\mathbb{Z}))$, we can identify $\sigma$ with an element of 
	\begin{align*}
		L^{p'}(\Omega;C^0(\mathbb{R}/ 2\pi\mathbb{Z}))^*\approx L^p_w(\Omega;\mathcal M(\mathbb{R}/ 2\pi\mathbb{Z}))
	\end{align*}
	(for this duality result see e.g. \cite[Theorem~8.20.3]{edwards} -- recall the subscript $w$ stands for weak measurability). Thus for almost all $x\in\Omega$, $\sigma_x(\cdot):=\sigma(\cdot,x)\in \mathcal M(\mathbb{R}/ 2\pi\mathbb{Z})$, and
	\begin{align*}
		\left(\int_\Omega \norm{\sigma(\cdot,x)}_{\mathcal M(\mathbb{R}/ 2\pi\mathbb{Z})}^p\, dx\right)^{\frac 1p} \leq \norm{\sigma}_* <\infty.
	\end{align*}
	This completes the proof of \eqref{eq30.2} and \eqref{eq:sigmaLp}.
\end{proof}

\begin{proof}[Proof of Proposition \ref{BimplyC}]
	For any open set $U\subset\subset\Omega$, the fact that $m$ satisfies the kinetic equation \eqref{eqinta3} and $\sigma$ from  Lemma \ref{lem:sigmaLp} forms the kinetic measure for $m$ in $U$ follows in a reasonably straightforward way; see Subsection 3.1 in \cite{GL}. The additional structure of $\sigma$ in \eqref{eq30.2} and \eqref{eq:sigmaLp} is also given in Lemma \ref{lem:sigmaLp}.
\end{proof}

\begin{rem}
	Under the assumptions of Theorem \ref{C2}, define the ``grand measure'' $\mu_m$ for $m$ to be the least upper bound of all entropy measures (see \cite[Definition 1.68]{ambrosio}), i.e.
		\begin{equation}
			\label{eqinta2}
			\mu_m:=\bigvee_{\Phi\in ENT, \|\Phi \|_{C^2}\leq 1} \|\dv \Phi(m)\|.
		\end{equation} 
		Then calling 
		\begin{equation*}
			(B')\,\, \text{The grand measure }\mu_m\in L^p_{\loc}(\Omega),
		\end{equation*}
		a little more work using the proof of Proposition \ref{LL16} and the estimate \eqref{eq621.2} shows that $(A)\Longrightarrow (B')$. And since $(B')$ is stronger than $(B)$ we have $(B')\Longrightarrow (C)$ from Proposition \ref{BimplyC}. Therefore combining these with Proposition \ref{LL1} in the next subsection, we can replace $(B)$ by $(B')$ in Theorem \ref{C2}, and in particular $(B)$ and $(B')$ are equivalent. This is analogous to part of the main theorem of \cite{GL}. 
\end{rem}

%
%
%

\subsection{Proof of $(C)\Longrightarrow (A)$ in Theorem \ref{C2}}

\begin{prop}
	\label{LL1}
	Let $m:\Omega\to\R^2$ satisfy \eqref{ageq3}. Assume $m$ satisfies the $L^{p}$ kinetic equation in the sense of Definition \ref{def:KIN} for some $1<p\leq \frac{4}{3}$, then $m\in B^{\frac{1}{3}}_{3p,\infty,\loc}(\Omega)$ .
\end{prop}

We start by introducing the analog of $\Delta$ used in Subsection 3.2 in \cite{GL}, which provides the right coercivity estimate in our case; see \eqref{cpeq30} below. To this end, we first define, for $\alpha>0$, the function $\varphi_{\alpha}(t):\R\to\R$ which satisfies
\begin{align}
\label{eqki2}
	&\bullet\;\varphi_{\alpha}\text{ is odd and }\pi\text{-periodic};\nn\\
	&\bullet\;\varphi_{\alpha}(t)=t^{\alpha}\text{ for }t\in [0,\pi/4];\nn\\
	&\bullet\;\varphi_{\alpha}> 0\text{ in }(0,\pi/2)\text{ and is smooth in }(0,\pi).
\end{align}
In the following we will take $\alpha=3p-3$ so that $0<\alpha\leq 1$ for $1<p\leq \frac 4 3$. 
In the sequel the functions will be defined in terms of $m$. Since $m$ is fixed, we will not explicitly show this dependence. We largely follow the notation of Section 3 in \cite{GL}.  Define $\theta(x)$ to be
\begin{equation}
	\label{eq:theta}
	\theta: \Omega\to [0,2\pi),\qd m(x)=e^{i\theta(x)}\text{ for a.e. }x\in\Omega.
\end{equation}
Next, for $(x,t)\in\Omega\times \R/ 2\pi \mathbb{Z}$ we define
\begin{equation}
	\label{eq:chi} \chi\lt(x,t\rt):=\mathds{1}_{e^{it}\cdot m(x)>0}. 
\end{equation}	
Finally, for $h\in\R$ and $e\in\mathbb{S}^1$, we define 
\begin{align}
	\label{cpeq10}
	\Delta_\alpha\lt(x,h,e\rt) :=\int_0^{2\pi}\int_0^{2\pi}\varphi_\alpha(s-t)\sin(s-t)D^{he}\chi(x,t)D^{he}\chi(x,s)\, dt\,ds,
\end{align}
where  $D^{he}\chi(x,\cdot)=\chi(x+he,\cdot)-\chi(x,\cdot)$. Our next lemma plays the role of \cite[Lemma 3.8]{GL}.

\begin{lem}
\label{LB16}
	Let $m:\Omega\to\mathbb{S}^1$ and $\alpha>0$. For all $x\in\Omega$, $e\in\mathbb{S}^1$ and $|h|<\mathrm{dist}\{x,\partial\Omega\}$, we have
	\begin{equation}
		\label{cpeq30}
		\Delta_{\alpha}(x,h,e) \geq c(\alpha)\, |D^{he}m(x)|^{3+\alpha}.
	\end{equation}
\end{lem}

\begin{proof}
	Since the proof follows almost the same lines of the proof of \cite[Lemma 3.8]{GL}, we only sketch the main ingredients focusing on the differences. Given $x\in\Omega$, $e\in\mathbb{S}^1$ and $|h|<\mathrm{dist}\{x,\partial\Omega\}$, exactly the same arguments as those at the beginning of the proof of \cite[Lemma 3.8]{GL} show that it is sufficient to assume $m(x+he)=e^{-i\beta}$ and $m(x)=e^{i\beta}$ for some $\beta\in[0,\pi/2]$. Thus \eqref{cpeq30} becomes
	\begin{equation*}
		\Xi_{\alpha}\lt(e^{-i\beta}, e^{i\beta}\rt)\geq c(\alpha)\lt(2\sin(\beta)\rt)^{3+\alpha}
	\end{equation*}
	for
	\begin{align*}
		\Xi_{\alpha}\lt(e^{-i\beta}, e^{i\beta}\rt):=&\int_0^{2\pi}\int_0^{2\pi}\varphi_\alpha(s-t)\sin(s-t)\\
		&\cdot\lt(\mathds{1}_{e^{it}\cdot e^{-i\beta}>0}-\mathds{1}_{e^{it}\cdot e^{i\beta}>0}\rt)\lt(\mathds{1}_{e^{is}\cdot e^{-i\beta}>0}-\mathds{1}_{e^{is}\cdot e^{i\beta}>0}\rt)\, dt\,ds.
	\end{align*}
	It is in turn sufficient to show
	\begin{equation}
		\label{cpeq146}
		\Xi_{\alpha}\lt(e^{-i\beta}, e^{i\beta}\rt)\geq \ti{c}(\alpha)\beta^{3+\alpha}.
	\end{equation}
	
	Following carefully the calculations in \cite[Lemma~3.8]{GL},  we have  
	\begin{equation}
		\label{eqgn1}
		\Xi_{\alpha}\lt(e^{-i\beta}, e^{i\beta}\rt) = 8\int_0^{2\beta}\varphi_{\alpha}(\omega)(2\beta-\omega)\sin(\omega)\, d\omega\qd\text{ for }\beta\in [0,\pi/4]
	\end{equation}
	and 
	\begin{align*}
		\Xi_{\alpha}\lt(e^{-i\beta},e^{i\beta}\rt) & = 
		8\int_0^{\pi-2\beta}\varphi_\alpha(\omega)(2\beta-\omega)\sin(\omega)\,d\omega \nn\\
		&\quad + 8\int_{\pi-2\beta}^{\frac{\pi}{2}}\varphi_\alpha(\omega)(\pi-2\omega)\sin(\omega) \, d\omega\quad\text{ for }\beta\in [\pi/4,\pi/2].
	\end{align*}
	Recall that $\varphi_{\alpha}(\omega)=\omega^{\alpha}$ for $\omega\in[0,\pi/4]$. Thus for $\beta\in \lt[0,\pi/8\rt]$, we have (after a change of variable $\omega=2\beta v$)
	\begin{align}
		\label{cpeq145.1}
		\Xi_{\alpha}\lt(e^{-i\beta}, e^{i\beta}\rt)   &\overset{(\ref{eqgn1})}{=}8( 2\beta)^{2+\alpha}\int_0^{1}v^{\alpha}(1-v)\sin (2\beta v)\, dv \nn\\
		&\gtrsim  (2\beta)^{3+\alpha}\int_0^1 v^{1+\alpha}(1-v)\,dv \gtrsim\beta^{3+\alpha}.
	\end{align}
	For $\beta\in [\pi/8,\pi/4]$, recalling \eqref{eqki2}, it is clear that
	\begin{equation*} 
		\Xi_{\alpha}\lt(e^{-i\beta}, e^{i\beta}\rt)\geq8\int_{0}^{\frac{\pi}{4}}\varphi_\alpha(\omega)\lt(\frac\pi 4-\omega\rt)\sin(\omega) \, d\omega=c_1>0.
	\end{equation*}
	 For $\beta\in[\pi/4,\pi/2]$, it is clear that $2\beta-\omega\geq \frac{\pi}{2}-\omega$ and $\pi-2\omega\geq \frac{\pi}{2}-\omega$ for all $\omega\in[0,\pi/2]$. Thus we have
	 \begin{equation*} 
	 	\Xi_{\alpha}\lt(e^{-i\beta}, e^{i\beta}\rt)\geq8\int_{0}^{\frac{\pi}{2}}\varphi_\alpha(\omega)\lt(\frac{\pi}{2}-\omega\rt)\sin(\omega) \, d\omega=c_2>0.
	 \end{equation*}
	 It follows that  $\Xi_{\alpha}\lt(e^{-i\beta}, e^{i\beta}\rt) \geq \min\{c_1,c_2\} \gtrsim \beta^{3+\alpha}$ for $\beta\in[\pi/8,\pi/2]$. This together with \eqref{cpeq145.1} establishes (\ref{cpeq146}) and hence \eqref{cpeq30}.
\end{proof}

Next we need a technical lemma.

\begin{lem}
\label{auxteclem1}
Let $m:\Omega\to\mathbb{S}^1$ and $q\in L^{\infty}\lt(\mathbb{R}/ 2\pi \mathbb{Z}\rt)$. Recalling $\chi(x,t)$ defined in \eqref{eq:chi}, for any $x_0,x_1\in \Omega$ we have
\begin{equation}
	\label{eqaux1}
	\lt|\int_0^{2\pi} \lt(\chi(x_0,s)-\chi(x_1,s) \rt) q(s)\, ds\rt|\lesssim\|q\|_{L^{\infty}\lt(\mathbb{R}/ 2\pi \mathbb{Z} \rt)}\lt|m(x_0)-m(x_1)\rt|.
\end{equation}
In addition, denoting by $\chi_{\ep}(x,t)$ the convolution of $\chi(x,t)$ in the $x$ variable, i.e. $\chi_{\ep}=\chi\ast\rho_{\ep}$ for a smooth approximation of the identity $\rho_{\ep}(x)$, we have for $x\in\Omega$ and $\ep<\mathrm{dist}\{x,\partial\Omega\}$
\begin{equation}
	\label{eqaux2}
	\lt|\int_0^{2\pi} \lt(\chi_{\ep}(x,s)-\chi(x,s) \rt) q(s)\, ds\rt|\lesssim\|q\|_{L^{\infty}\lt(\mathbb{R}/ 2\pi \mathbb{Z} \rt)}\Xint-_{B_{\ep}(x)}|m(z)-m(x)|\,dz.
\end{equation}
\end{lem}
\begin{proof} 
	Note that for any $a\in \mathbb{R}$ we have $\min\lt\{\lt|a-2\pi k\rt|:k\in \mathbb{Z}\rt\}\leq \pi$. So given $x_0,x_1\in\Omega$ there exists $k_0=k_0(x_0,x_1)\in \mathbb{Z}$ such that $\lt|\theta(x_1)-\theta(x_0)-2\pi k_0\rt|\leq \pi$, where $\theta$ satisfies \eqref{eq:theta}.
Define $\wt{\theta}(x_0)=\theta(x_0)+2\pi k_0$, so that $|\wt{\theta}(x_0)-\theta(x_1)|\leq \pi$. It is easily checked that $2-2\cos(t)\gtrsim t^2$ for $|t|\leq\pi$. It follows that
\begin{equation}
\label{eqaux11}
\lt|\wt{\theta}(x_0)-\theta(x_1) \rt|^2\lesssim 2-2\cos\lt(\wt{\theta}(x_0)-\theta(x_1)\rt)=\lt|m(x_0)-m(x_1)\rt|^2.
\end{equation}
As the function $q$ is $2\pi$-periodic, we have
\begin{align*}
\lt|\int_0^{2\pi} \lt(\chi(x_0,s)-\chi(x_1,s) \rt) q(s) ds\rt|&=\lt|\int_{\wt{\theta}(x_0)-\frac{\pi}{2}}^{\wt{\theta}(x_0)+\frac{\pi}{2}} q(s)\, ds
-\int_{\theta(x_1)-\frac{\pi}{2}}^{\theta(x_1)+\frac{\pi}{2}} q(s)\, ds\rt|\nn\\
&\lesssim \|q\|_{L^{\infty}\lt(\mathbb{R}/ 2\pi \mathbb{Z} \rt)}\lt|\wt{\theta}(x_0)-\theta(x_1)\rt|\nn\\
&\overset{(\ref{eqaux11})}{\lesssim} \|q\|_{L^{\infty}\lt(\mathbb{R}/ 2\pi \mathbb{Z} \rt)}\lt|m(x_0)-m(x_1)\rt|,
\end{align*}
which establishes (\ref{eqaux1}). 

Next, using $\int_{\R^2}\rho_{\ep}=1$ and $\rho_{\ep}\lesssim 1/\ep^2$, we find
\begin{align*}
	&\lt|\int_0^{2\pi} \lt(\chi_{\ep}(x,s)-\chi(x,s) \rt) q(s) ds\rt|\\
	&\qd\qd=\lt|\int_0^{2\pi}\lt(\int_{B_{\ep}(x)}\rho_{\ep}(x-z) \lt(\chi(z,s)-\chi(x,s) \rt) dz\rt)q(s) \,ds\rt|\\
	&\qd\qd=\lt|\int_{B_{\ep}(x)}\rho_{\ep}(x-z)\int_0^{2\pi} \lt(\chi(z,s)-\chi(x,s) \rt) q(s) \,ds\,dz\rt|\\
	&\qd\qd\overset{\eqref{eqaux1}}{\lesssim}\|q\|_{L^{\infty}(\mathbb{R}/ 2\pi \mathbb{Z})}\cdot\frac{1}{\ep^2}\int_{B_{\ep}(x)}|m(z)-m(x)|\,dz,
\end{align*}
which establishes \eqref{eqaux2}.
%
\end{proof}

\begin{proof}[Proof of Proposition \ref{LL1}]
Following \cite{GL} we introduce the regularization of $\Delta_{\alpha}$ defined in \eqref{cpeq10} in the $x$ variable. Specifically, for $x\in\Omega, e\in\mathbb{S}^1$ and $|h|<\mathrm{dist}\{x,\partial\Omega\}$ define
\begin{align*}
\Delta^{\ep}_\alpha\lt(x,h,e\rt) & :=\int_0^{2\pi}\int_0^{2\pi}\varphi_\alpha(s-t)\sin(s-t)D^{he}\chi_{\ep}(x,t)D^{he}\chi_{\ep}(x,s)\, dt\,ds,
\end{align*}
for $\ep>0$ sufficiently small, where recall that $\chi_{\ep}=\chi\ast\rho_{\ep}$ for a smooth approximation of the identity $\rho_{\ep}(x)$. Assume without loss of generality that $e=e_1$, and the general case can be dealt with by rotation. In the sequel we omit the dependence of all the quantities on $e$. For all $x\in\Omega$ and $|h|+\ep<\mathrm{dist}\{x,\partial\Omega\}$, the computations in \cite[Lemma 3.9]{GL} using the kinetic equation \eqref{eqinta3} give
\begin{equation}
	\label{partial_h}
	\partial_h \Delta_{\alpha}^{\ep}(x,h)=I_{\alpha}^{\ep}(x,h)+\na\cdot A_{\alpha}^{\ep}(x,h), 
\end{equation}
where, denoting $\sigma_{\ep}=\sigma\ast\rho_{\ep}$ and $\chi_{\ep}^h(x,s)=\chi_{\ep}(x+he_1,s)$ with the same meaning extended to all other applicable functions, we have
\begin{align}
\label{cpeq297}
I_{\alpha}^{\e}(x,h) &=-2\int_0^{2\pi}\sigma_\e^h(x,t) \int_0^{2\pi} \varphi'_{\alpha}(s-t)\chi_\e(x,s)\sin(s) \, ds\, dt \nn\\
&\quad + 2\int_0^{2\pi} \sigma_\e(x,t) \int_0^{2\pi} \varphi'_{\alpha}(s-t)\chi_\e^h(x,s)\sin(s) \, ds\, dt,
\end{align}
and
\begin{align*}
A_{\alpha,1}^{\e}(x,h) &= 2\iint_{\lt[0,2\pi\rt]\times \lt[0,2\pi\rt]}  \varphi_{\alpha}(s-t)\sin(s)\cos(t) \,\chi_\e^h(x,s) D^h\chi_\e(x,t)\, ds\, dt, \nn\\
A_{\alpha,2}^{\e}(x,h)&= 2\iint_{\lt[0,2\pi\rt]\times \lt[0,2\pi\rt]} \varphi_{\alpha}(s-t)\sin(s) \sin(t) \,\chi_\e(x,s) \chi^h_\e(x,t) \, ds\, dt.
\end{align*}
In the following, we consider an arbitrary fixed domain $U\subset\subset \Omega$ so that there exist intermediate domains $U\subset\subset\Omega'\subset\subset\Omega''\subset\subset\Omega$. \nl

\noindent\em Step 1. \rm We will show that for any fixed $|h|<\mathrm{dist}\{\Omega',\partial\Omega''\}$ and a.e.\ $x\in \Omega'$, we have (recalling $\theta(x)$ given in \eqref{eq:theta})
\begin{equation}
\label{eqokl62}
\lim_{\ep\rightarrow 0} I^{\ep}_{\alpha}(x,h)=I_{\alpha}(x,h),
\end{equation}
\begin{align}
	\label{cpeq298.2}
	\lim_{\ep\rightarrow 0} A_{\alpha,i}^{\e}(x,h) &= A_{\alpha,i}(x,h)\qd\text{ for }i=1,2,
\end{align} 
and
\begin{equation}
	\label{eqokl50.3}
	\lim_{\ep \rightarrow 0} \Delta^{\ep}_{\alpha}(x,h)=\Delta_{\alpha}(x,h),
\end{equation}
where 
\begin{align}
\label{cpeq18}
I_{\alpha}(x,h)=&-2\int_0^{2\pi} \lt(\int_{\theta(x)-\frac{\pi}{2}}
^{\theta(x)+\frac{\pi}{2}} \varphi'_{\alpha}(s-t)\sin(s) \, ds\rt)  d\sigma_x^h(t)\nn\\
&+ 2\int_0^{2\pi}\lt(  \int_{\theta^h(x)-\frac{\pi}{2}}
^{\theta^h(x)+\frac{\pi}{2}}  \varphi'_{\alpha}(s-t)\sin(s) \, ds\rt)  d\sigma_x(t),
\end{align}
\begin{align}
	\label{eq:A_alpha}
	A_{\alpha,1}(x,h) &= 2\iint_{\lt[0,2\pi\rt]\times \lt[0,2\pi\rt]}  \varphi_{\alpha}(s-t)\sin(s)\cos(t) \,\chi^h(x,s) D^h\chi(x,t)\, ds\, dt, \nn\\
	A_{\alpha,2}(x,h)&=2\iint_{\lt[0,2\pi\rt]\times \lt[0,2\pi\rt]} \varphi_{\alpha}(s-t)\sin(s) \sin(t) \,\chi(x,s) \chi^h(x,t) \, ds\, dt,
\end{align}
and $\Delta_{\alpha}(x,h)=\Delta_{\alpha}(x,h,e_1)$ is given in \eqref{cpeq10}.\nl

\noindent\em Proof of Step 1. \rm For any  $\psi\in C^0\lt(\mathbb{R}/ 2\pi\mathbb{Z}\rt)$, define $\mathcal{F}_{\psi}(x):=\int_0^{2\pi}\psi(t)\,d\sigma_x(t)$ for a.e. $x\in\Omega$. It follows that
 \begin{equation}
 	\label{eq:F}
 \lt|\FI_{\psi}(x)\rt|\leq \|\psi\|_{C^0(\R/2\pi\mathbb{Z})}\nu(x)\qd\text{ for a.e. }x\in\Omega\text{ and all }\psi\in C^0\lt(\mathbb{R}/ 2\pi\mathbb{Z}\rt),
 \end{equation}
  where recall from \eqref{eq:sigmaLp} that $\nu(x)=\|\sigma_x\|_{\mathcal{M}(\R/2\pi\mathbb{Z})}\in L^p_{\loc}(\Omega)$. Thus $\FI_{\psi}\in L^p_{\loc}(\Omega)$ for all $\psi\in C^0\lt(\mathbb{R}/ 2\pi\mathbb{Z}\rt)$. By definition of $\sigma_{\ep}$ and \eqref{eq30.2} we have
\begin{align}
\label{eqokl61}
 &\int_0^{2\pi}\psi(t) \sigma_{\ep}(x,t)\, dt= \int_0^{2\pi}\psi(t)\int_{B_{\ep}(x)}\rho_{\ep}(x-z)\,d\sigma(z,t)\nn\\
 &\qd\qd\qd\qd\qd\qd=\int_{B_\e(x)}\lt(\int_0^{2\pi}\psi(t)\,d\sigma_z(t)\rt)\rho_\e(x-z)\,dz=\int_{B_{\e}(x)}\FI_{\psi}(z)\rho_\e(x-z)\,dz.
\end{align}
Combining \eqref{eqokl61} with \eqref{eq:F}, we obtain that, for all Lebesgue points $x$ of $\nu$ with $\nu(x)<\infty$, there exists $\ep_0(x)>0$ such that
\begin{equation*}
	\lt|\int_0^{2\pi}\psi(t) \sigma_{\ep}(x,t)\, dt\rt|\leq \|\psi\|_{C^0(\R/2\pi\mathbb{Z})}\int_{B_{\e}(x)}\nu(z)\rho_\e(x-z)\,dz \leq C\|\psi\|_{C^0(\R/2\pi\mathbb{Z})}\nu(x)
\end{equation*}
for all $\psi\in C^0\lt(\mathbb{R}/ 2\pi\mathbb{Z}\rt)$ and all  $\ep\in \lt(0,\ep_0(x)\rt)$. It follows that
\begin{equation}
	\label{control:sigma_e} \|\sigma_\e(x,\cdot)\|_{L^1\lt(\mathbb{R}/ 2\pi\mathbb{Z}\rt)}\leq  C\nu(x)<\infty\qd\text{  for all }\ep \in \lt(0,\ep_0(x)\rt).
\end{equation}	

Let $\{\psi_j\}_j$ be a countable dense subset of  $C^0\lt(\mathbb{R}/ 2\pi\mathbb{Z}\rt)$, and  $\Lambda\subset\Omega$ be the intersection of the Lebesgue points of $\FI_{\psi_j}$ for all $j$ and the Lebesgue points of $\nu$ with $\nu(x)<\infty$. We know $|\Omega\setminus\Lambda|=0$. It follows from \eqref{eqokl61} that
\begin{equation}
	\label{eq:sigma}
	\lim_{\ep\to 0}\int_0^{2\pi} \psi_j(t)\sigma_\e(x,t)\,dt=\FI_{\psi_j}(x)=\int_0^{2\pi}\psi_j(t)\,d \sigma_x(t)\qd\text{ for all }j\text{ and all }x\in\Omega'\cap\Lambda. 
\end{equation}
By density of $\{\psi_j\}_j$ in $C^0\lt(\mathbb{R}/ 2\pi\mathbb{Z}\rt)$ and \eqref{control:sigma_e},  one can extend \eqref{eq:sigma} to all $\psi\in C^0\lt(\mathbb{R}/ 2\pi\mathbb{Z}\rt)$ to deduce that
\begin{equation}
	\label{eqokl41}
	\sigma_{\ep}(x,\cdot)\mathcal{L}^1\rightharpoonup \sigma_{x}(\cdot)\qd\text{ in }\mathcal{M}\lt(\mathbb{R}/ 2\pi\mathbb{Z}\rt)\text{ as }\ep\rightarrow 0\text{ for all }x\in \Omega'\cap \Lambda.
\end{equation}

We will use (\ref{eqokl41}) to first find the limit of the second term of $I^{\ep}_{\alpha}$ (recalling (\ref{cpeq297})) as $\ep\rightarrow 0$.
By \eqref{eqaux2} of Lemma \ref{auxteclem1}, for all $t\in\mathbb{R}/ 2\pi\mathbb{Z}$ we have
\begin{align*}
&\lt|\int_0^{2\pi} \varphi_{\alpha}'(s-t)\sin(s)\chi^h(x,s) ds-
\int_0^{2\pi}\varphi_{\alpha}'(s-t)\sin(s)\chi_{\ep}^h(x,s) ds\rt|\nn\\
&\qd \lesssim \|\varphi_{\alpha}'\|_{C^0\lt(\mathbb{R}/ 2\pi\mathbb{Z}   \rt)}\Xint-_{B_\e(x)}|m^h(z)-m^h(x)|\,dz, \nn
\end{align*}
and hence for all Lebesgue points $x$ of $m^h$ we have
\begin{equation*}
\int_0^{2\pi} \varphi_{\alpha}'(s-t)\sin(s)\chi_\e^h(x,s) ds\rightarrow \int_0^{2\pi} \varphi_{\alpha}'(s-t)\sin(s)\chi^h(x,s) ds\qd\text{ in }C_t^0\lt(\mathbb{R}/ 2\pi \mathbb{Z} \rt).
\end{equation*}
This together with \eqref{control:sigma_e} and \eqref{eqokl41} shows 
\begin{align*}
&\lim_{\ep\rightarrow 0} \iint_{\lt[0,2\pi\rt]\times \lt[0,2\pi\rt]}  \varphi_{\alpha}'(s-t)\sin(s)\chi_{\ep}^h(x,s) \sigma_{\ep}(x,t) ds \, dt\nn\\
&\qd\qd\qd\qd=\iint_{\lt[0,2\pi\rt]\times \lt[0,2\pi\rt]}  \varphi_{\alpha}'(s-t)\sin(s) \chi^h(x,s)ds \, d\sigma_{x}(t)\qd\text{ for a.e. }x\in \Omega'.
\end{align*}
In exactly the same way we can deal with the first term in $I^{\ep}_{\alpha}$ and hence establish
(\ref{eqokl62}). Finally, \eqref{cpeq298.2} and \eqref{eqokl50.3} follow from \eqref{eqaux2} and straightforward estimates. This completes the proof of Step 1. \nl

\noindent\em Proof of Proposition \ref{LL1} completed. \rm Let
\begin{equation}
	\label{alpha}
	\alpha=3p-3.
\end{equation}	
We take a smooth nonnegative cut off function $\gamma\in C^{\infty}_c(\Omega')$ with $\|\gamma\|_{L^\infty(\Omega)}\leq 1$ and $\gamma\equiv 1$ on $\overline U$. Using \eqref{partial_h}, the Dominated Convergence Theorem and Step 1,  we have for all $|h|<\mathrm{dist}\{\Omega',\partial\Omega''\}$ that
\begin{align}
\label{eqfgg1}
&\int_{\Omega} \gamma(x) \Delta_{\alpha}(x,h) \, dx\nn\\
&\qd\overset{(\ref{eqokl50.3})}{=} 
\lim_{\ep\rightarrow 0}\int_{\Omega} \gamma(x) \Delta^{\ep}_{\alpha}(x,h) \, dx\nn\\
&\qd\overset{(\ref{partial_h})}{=}\lim_{\ep\rightarrow 0} \lt(\int_{0}^h \int_{\Omega} \gamma(x) I^{\ep}_{\alpha}(x,\hi) dx\, d\hi- \int_{0}^h \int_{\Omega} \na\gamma(x)\cdot  A_{\alpha}^{\ep}(x,\hi) dx\, d\hi \rt)\nn\\
&\qd\overset{(\ref{eqokl62}),(\ref{cpeq298.2})}{=}\int_{0}^h \int_{\Omega} \gamma(x) I_{\alpha}(x,\hi) dx\, d\hi- \int_{0}^h \int_{\Omega} \na\gamma(x)\cdot  A_{\alpha}(x,\hi) dx\, d\hi.
\end{align}
In the following Lemmas \ref{lem:I_alpha} and \ref{lem:A_alpha}, we will show
\begin{align}
	\label{cpeq283}
	\lt|\int_{\Omega} I_{\alpha}(x,h)\,\gamma(x)\, dx\rt|\leq C\lt(\lt\lVert\gamma\lt(\lt|D^h m\rt|^{\alpha}+\lt|D^{-h}m\rt|^{\alpha}\rt)\rt\rVert_{L^{p'}(\Omega)}+\lt\lVert\na \gamma\rt\rVert_{L^{\infty}(\Omega)}|h|\rt)\|\nu\|_{L^{p}(\Omega'')}
\end{align}
and
\begin{equation}
\label{estA_alpha}
\lt|\int_{\Omega}A_{\alpha}(x,h) \cdot  \na \gamma(x)\,dx \rt|\leq
C\int_{\Omega}\lt|D^h m\rt|\lt|\na\gamma\rt|\,dx.
\end{equation}
We combine the lower bound established in (\ref{cpeq30}) of Lemma \ref{LB16} with the upper bounds \eqref{cpeq283} and \eqref{estA_alpha} and apply them to (\ref{eqfgg1}) to deduce that
\begin{align}
\label{eqokl105}
\int_{\Omega} \gamma(x) \lt|D^hm(x)\rt|^{3+\alpha} \, dx &\leq C\lt|h\rt|\sup_{|\hi|\in \lt[0,|h|\rt]}\lt\lVert\lt(\lt|D^{\hi} m\rt|^{\alpha}+\lt|D^{-\hi} m\rt|^{\alpha} \rt)\gamma\rt\rVert_{L^{p'}(\Omega)} \lt\lVert\nu\rt\rVert_{L^{p}(\Omega'')}\nn\\
&+C\lt|h\rt|^2\lt\lVert\na \gamma\rt\rVert_{L^{\infty}(\Omega)}\lt\lVert\nu\rt\rVert_{L^{p}(\Omega'')}\nn\\
&+C\lt|h\rt|\lt\lVert\na\gamma\rt\rVert_{L^{\infty}(\Omega)}\sup_{|\hi|\in \lt[0,|h|\rt]}\int_{\Omega'} \lt|D^{\hi} m\rt|\, dx.
\end{align}

Recalling \eqref{alpha}, we have $\alpha p'=3+\alpha=3p$,
and since $0\leq\gamma\leq 1$, we have
\begin{align*}
\lt\lVert\gamma\lt|D^h m\rt|^{\alpha}\rt\rVert_{L^{p'}(\Omega)}=\lt(\int_{\Omega} \gamma^{p'} \lt|D^h m\rt|^{3p} \, dx\rt)^{\frac{1}{p'}}\leq \lt(\int_{\Omega} \gamma \lt|D^h m\rt|^{3p} \, dx\rt)^{\frac{1}{p'}}.
\end{align*}
Thus it follows from Young's inequality that
\begin{align}
\label{cpeq303}
&\lt|h\rt| \sup_{\lt\{\hi:\lt|\hi\rt|\leq \lt|h\rt|\rt\}}\lt\lVert\gamma\lt|D^{\hi} m\rt|^{\alpha}\rt\rVert_{L^{p'}(\Omega)}\nn\\
&\qd \leq \lt|h\rt|\sup_{\lt\{\hi:\lt|\hi\rt|\leq \lt|h\rt|\rt\}}\lt(\int_{\Omega} \gamma \lt|D^{\hi} m\rt|^{3p} \, dx\rt)^{\frac{1}{p'}}\nn\\
&\qd\leq C(\delta)|h|^p+\delta\sup_{\lt\{\hi:\lt|\hi\rt|\leq \lt|h\rt|\rt\}}\lt(\int_{\Omega} \gamma \lt|D^{\hi} m\rt|^{3p} \, dx\rt),\qd\forall \delta>0.
\end{align}
Next by \cite[Theorem 2.6]{GL}, we know $m\in B^{\frac 13}_{3,\infty,\loc}(\Omega)$. Thus, for all $|h|$ sufficiently small we have
\begin{equation}
	\label{estlastterm}
	\int_{\Omega'}\lt|D^hm\rt|\,dx\leq C \lt\lVert D^h m\rt\rVert_{L^3(\Omega')}\leq C |h|^{\frac 13}.
\end{equation}
Putting \eqref{cpeq303} and \eqref{estlastterm} into \eqref{eqokl105} and using $3+\alpha=3p$, we readily deduce
\begin{align*}
\int_{\Omega}\gamma \lt|D^{h} m\rt|^{3p}\, dx  &\leq
C\lt(C(\delta)|h|^p+\delta\sup_{\lt\{\hi:\lt|\hi\rt|\leq \lt|h\rt|\rt\}}\lt(\int_{\Omega} \gamma \lt|D^{\hi} m\rt|^{3p} \, dx\rt)\rt) \lt\lVert\nu\rt\rVert_{L^{p}(\Omega'')}\nn\\
&\qd+C\lt|h\rt|^2\lt\lVert\na \gamma\rt\rVert_{L^{\infty}(\Omega)}\lt\lVert\nu\rt\rVert_{L^{p}(\Omega'')}+ C|h|^{\frac 43}\lt\lVert\na \gamma\rt\rVert_{L^{\infty}(\Omega)}.
\end{align*}
Taking $\delta>0$ sufficiently small and recalling $p\in \lt(1,\frac 43\rt]$, we obtain, for all $0<\xi<\mathrm{dist}\{\Omega',\partial\Omega''\}$, 
\begin{align*}
\sup_{\lt\{h:\lt|h\rt|\leq \xi\rt\}} \int_{\Omega}\gamma\lt|D^{h} m\rt|^{3p}\, dx & \leq \frac{1}{2}\sup_{\lt\{h:\lt|h\rt|\leq \xi\rt\}}\sup_{\lt\{\hi:\lt|\hi\rt|\leq \lt|h\rt|\rt\}}\int_{\Omega} \gamma\lt|D^{\hi} m\rt|^{3p} dx +C \sup_{\lt\{h:\lt|h\rt|\leq \xi\rt\}}\lt|h\rt|^{p}\nn\\
&=\frac{1}{2}\sup_{\lt\{h:\lt|h\rt|\leq \xi\rt\}}\int_{\Omega} \gamma\lt|D^{h} m\rt|^{3p} dx+C \xi^{p}.
\end{align*}
Recalling $\gamma\equiv 1$ on $\overline U$, we deduce from the above that
\begin{equation*}
\sup_{\lt\{h:\lt|h\rt|\leq \xi\rt\}} \int_{U}\gamma \lt|D^{h} m\rt|^{3p}\, dx \leq  C\xi^{p}\qd\text{ for all }0<\xi<\mathrm{dist}\{\Omega',\partial\Omega''\}.
\end{equation*}
For larger values of $\xi$, the boundedness of $m$ implies that $\xi^{-\frac 13}\sup_{|h|\leq \xi}\lVert D^h m\rVert_{L^{3p}(U)}$ is bounded. It follows that $m\in B^{\frac{1}{3}}_{3p,\infty}(U)$ for all $U\subset\subset \Omega$. 
\end{proof}

\begin{lem}
	\label{lem:I_alpha}
	Let $I_{\alpha}(x,h)$ be given in \eqref{cpeq18}. For domains $U\subset\subset\Omega'\subset\subset\Omega''\subset\subset\Omega$ and any $\gamma\in C^{\infty}_c(\Omega')$, the estimate \eqref{cpeq283} holds for all $|h|<\mathrm{dist}\{\Omega',\partial\Omega''\}$.
\end{lem}

To simplify notation, define 
\begin{equation}
	\label{cpeq14.3}
	\WI_{\alpha}(x,s):=\int_0^{2\pi} \varphi_{\alpha}(s-t) d\sigma_x(t).
\end{equation}
It is clear that
\begin{equation}
	\label{cpeq14.4}
	\|\WI_{\alpha}(x,\cdot)\|_{L^{\infty}(\mathbb{R}/ 2\pi\mathbb{Z})}\leq \|\varphi_{\alpha}\|_{C^0}\,\nu(x)\qd\text{ for a.e.  }x\in\Omega,
\end{equation}
where recall that $\nu(x)=\|\sigma_x\|_{\mathcal{M}(\mathbb{R}/ 2\pi \mathbb{Z})}\in L^p_{\loc}(\Omega)$. We first do some computations. 

\begin{lem}
	Under the assumptions of Lemma \ref{lem:I_alpha}, we have
	\begin{equation}
		\label{cpeq201.5}
		\int_{\Omega} I_{\alpha}(x,h) \gamma(x) dx=2\Xi_1(h)-4\Xi_2(h),
	\end{equation}
	where
	\begin{align}
		\label{cpeq172}
		\Xi_1(h)&=\int_{\Omega} \lt(\int_{\theta(x)-\frac{\pi}{2}}^{\theta(x)+\frac{\pi}{2}}  \WI_{\alpha}^h\lt(x,s\rt)\cos(s) \, ds -\int_{\theta^h(x)-\frac{\pi}{2}}^{\theta^h(x)+\frac{\pi}{2}}   \WI_{\alpha}\lt(x,s\rt)\cos(s) \, ds\rt)\gamma(x) \,dx
	\end{align}
	and
	\begin{align}
		\label{cpeq173}
		\Xi_2(h)&=\int_{\Omega} \lt(\sin\lt( \theta(x)+\frac{\pi}{2}\rt)\WI_{\alpha}^h\lt(x,\theta(x)+\frac{\pi}{2}\rt) \rt.\nn\\
		&\qd\qd\qd\qd -\lt. \sin\lt( \theta^h(x)+\frac{\pi}{2}\rt) \WI_{\alpha}\lt(x,\theta^h(x)+\frac{\pi}{2}\rt)  \rt) \gamma(x) \,dx.
	\end{align}
\end{lem}

\begin{proof}
	Using integration by parts and $\pi$-periodicity of $\varphi_{\alpha}$ gives
	\begin{align}
		\label{cpeq14.2}
		&\int_{\theta(x)-\frac{\pi}{2}}^{\theta(x)+\frac{\pi}{2}} \varphi'_{\alpha}(s-t)\sin(s) \, ds\nn\\
		&=2\varphi_{\alpha}\lt(\theta(x)+\frac{\pi}{2}-t\rt)\sin\lt( \theta(x)+\frac{\pi}{2}\rt) -\int_{\theta(x)-\frac{\pi}{2}}^{\theta(x)+\frac{\pi}{2}}   \varphi_{\alpha}(s-t)\cos(s) \, ds.
	\end{align}
	Thus
	\begin{align}
		\label{cpeq12}
		&\int_0^{2\pi}\lt(\int_{\theta(x)-\frac{\pi}{2}}
		^{\theta(x)+\frac{\pi}{2}} \varphi'_{\alpha}(s-t)\sin(s) \, ds\rt)  d\sigma^h_x(t)\nn\\
		&\overset{(\ref{cpeq14.3}),(\ref{cpeq14.2})}{=}2\sin\lt( \theta(x)+\frac{\pi}{2}\rt)\WI_{\alpha}^h\lt(x,\theta(x)+\frac{\pi}{2}\rt)-\int_{\theta(x)-\frac{\pi}{2}}^{\theta(x)+\frac{\pi}{2}}   \WI_{\alpha}^h(x,s)   \cos(s) \, ds.
	\end{align}
	And in the same way we have 
	\begin{align}
		\label{cpeq13}
		&\int_{0}^{2\pi}\lt(\int_{\theta^h(x)-\frac{\pi}{2}}
		^{\theta^h(x)+\frac{\pi}{2}} \varphi'_{\alpha}(s-t)\sin(s) \, ds\rt)  d\sigma_x(t)\nn\\
		&=2\sin\lt( \theta^h(x)+\frac{\pi}{2}\rt) \WI_{\alpha}\lt(x,\theta^h(x)+\frac{\pi}{2}\rt) -\int_{\theta^h(x)-\frac{\pi}{2}}^{\theta^h(x)+\frac{\pi}{2}}   \WI_{\alpha}(x,s)\cos(s) \, ds.
	\end{align}
	Putting \eqref{cpeq18}, (\ref{cpeq12}), (\ref{cpeq13}), (\ref{cpeq172}), (\ref{cpeq173}) together establishes (\ref{cpeq201.5}). 
\end{proof}

\begin{proof}[Proof of Lemma \ref{lem:I_alpha}]
	Assume $1<p\leq \frac{4}{3}$, and thus $0<\alpha=3p-3\leq 1$. We first show
	\begin{align}
		\label{cpeq201.4}
		\lt|\Xi_1(h)\rt|&\leq  C\lt(\lt\lVert\gamma\lt(\lt|D^h m\rt|^{\alpha}+\lt|D^{-h}m\rt|^{\alpha}\rt)\rt\rVert_{L^{p'}(\Omega)}+\lt\lVert\na \gamma\rt\rVert_{L^{\infty}(\Omega)}|h|\rt)\|\nu\|_{L^{p}(\Omega'')}.
	\end{align}
	A change of variable gives
	\begin{align}
		\label{cpeq221.4}
		&\int_{\Omega} \lt(\int_{\theta(x)-\frac{\pi}{2}}^{\theta(x)+\frac{\pi}{2}}   \WI_{\alpha}^h\lt(x,s\rt)\cos(s) \, ds   \rt)\gamma(x) \, dx\nn\\
		&\qd\qd=\int_{\Omega} \lt(\int_{\theta^{-h}(x)-\frac{\pi}{2}}^{\theta^{-h}(x)+\frac{\pi}{2}}   \WI_{\alpha}\lt(x,s\rt)\cos(s) \, ds   \rt)\gamma(x-he_1) \, dx,
	\end{align}
	and thus
	\begin{align}
		\label{cpeq221.45}
		&\lt|  \int_{\Omega} \lt(\int_{\theta(x)-\frac{\pi}{2}}^{\theta(x)+\frac{\pi}{2}}   \WI_{\alpha}^h\lt(x,s\rt)\cos(s) \, ds   \rt)\gamma(x) \, dx   \rt.\nn\\
		&\qd\qd\qd-\lt.  \int_{\Omega} \lt(\int_{\theta^{-h}(x)-\frac{\pi}{2}}^{\theta^{-h}(x)+\frac{\pi}{2}}   \WI_{\alpha}\lt(x,s\rt)\cos(s) \, ds   \rt)\gamma(x) \, dx   \rt|\nn\\
		&\qd\overset{(\ref{cpeq221.4})}{=}\lt|\int_{\Omega} \lt(\int_{\theta^{-h}(x)-\frac{\pi}{2}}^{\theta^{-h}(x)+\frac{\pi}{2}}  \WI_{\alpha}\lt(x,s\rt) \cos(s) \, ds\rt)\lt(\gamma(x-he_1)-\gamma(x) \rt) \, dx \rt|\nn\\
		&\qd\overset{(\ref{cpeq14.4})}{\lesssim} \|\na \gamma\|_{L^{\infty}(\Omega)}\|\nu\|_{L^{p}\lt(\Omega''\rt)}|h|.
	\end{align}
	Therefore, recalling the definition of $\chi$ in \eqref{eq:chi}, we have
	\begin{align} 
		\label{cpeq233}
		&\lt|\Xi_1(h)-\int_{\Omega} \lt(\int_0^{2\pi} \lt(\chi^{-h}(x,s)-\chi^{h}(x,s)\rt) \mathcal{W_{\alpha}}(x,s)\cos(s) ds\rt)  \gamma(x)\, dx \rt|\nn\\
		&\qd \overset{(\ref{cpeq172}), (\ref{cpeq221.45}) }{\lesssim} \|\na \gamma\|_{L^{\infty}(\Omega)}\|\nu\|_{L^{p}\lt(\Omega''\rt)}|h|.
	\end{align}
	As $0<\alpha\leq 1$, it follows from \eqref{eqaux1} and \eqref{cpeq14.4} that
	\begin{align} 
		\label{cpeq233.4}
		&\lt|\int_0^{2\pi} \lt(\chi^{-h}(x,s)-\chi^h(x,s)\rt) \mathcal{W_{\alpha}}(x,s)\cos(s) ds\rt|\nn\\
		&\qd\qd\lesssim\nu(x) \lt|m(x+he_1)-m(x-he_1)\rt|\nn\\
		&\qd\qd\leq \nu(x)\lt(\lt|D^h m(x)\rt|^\alpha+\lt|D^{-h} m(x)\rt|^{\alpha}\rt)\qd\text{ for }a.e.\ x\in\Omega'. 
	\end{align}
	 Putting (\ref{cpeq233}) and (\ref{cpeq233.4}) together gives (\ref{cpeq201.4}). 
	 
	 Next we show
	 \begin{equation}
	 	\label{cpeq269}
	 	\lt|\Xi_2(h)\rt|\leq  C\lt(\lt\lVert\gamma\lt(\lt|D^h m\rt|^{\alpha}+\lt|D^{-h}m\rt|^{\alpha}\rt)\rt\rVert_{L^{p'}(\Omega)}+\lt\lVert\na \gamma\rt\rVert_{L^{\infty}(\Omega)}|h|\rt)\|\nu\|_{L^{p}(\Omega'')}.
	 \end{equation}
	 In the same way as (\ref{cpeq221.45}), from (\ref{cpeq173}) we have
	 \begin{equation}
	 	\label{cpeq173.4}
	 	\lt|\Xi_2(h)-\wt\Xi_2(h)\rt|\lesssim \|\na \gamma\|_{L^{\infty}(\Omega)}\|\nu\|_{L^{p}\lt(\Omega''\rt)}|h|,
	 \end{equation}
	 where
	 \begin{align*}
	 	\wt\Xi_2(h)&:=\int_{\Omega} \lt(\sin\lt( \theta^{-h}(x)+\frac{\pi}{2}\rt)\WI_{\alpha}\lt(x,\theta^{-h}(x)+\frac{\pi}{2}\rt)\rt.\nn\\
	 	&\qd\qd\qd\qd -\lt. \sin\lt( \theta^h(x)+\frac{\pi}{2}\rt)  \WI_{\alpha}\lt(x,\theta^h(x)+\frac{\pi}{2}\rt) \rt) \gamma(x) \,dx.
	 \end{align*}
 It is easy to check that for $0<\alpha\leq 1$ we have  
	 \begin{equation*}
	 	\lt|\varphi_\alpha(t_1)- \varphi_\alpha(t_2)\rt|\leq C(\alpha)\lt|t_1- t_2\rt|^{\alpha}\qd\text{ for all }
	 	t_1, t_2\in \R.
	 \end{equation*}
 	Thus
	 	\begin{align}
	 		\label{cpeq14.6}
	 		\lt|\WI_{\alpha}\lt(x,s_1\rt)-\WI_{\alpha}\lt(x,s_2\rt)\rt|&\overset{(\ref{cpeq14.3})}{\leq} C(\alpha)\lt|s_1-s_2\rt|^{\alpha}\nu(x)\qd\text{ for all }s_1,s_2\in\R\text{ and a.e. }x\in\Omega.
	 \end{align}
	 It follows that for $0<\alpha\leq 1$ we have
	 \begin{align*}
	 	&\lt|\sin\lt(s_1\rt)\WI_{\alpha}(x,s_1)- \sin\lt(s_2\rt)\WI_{\alpha}(x,s_2)\rt|\nn\\
	 	&\qd\qd\leq \lt|\WI_{\alpha}(x,s_1)-\WI_{\alpha}(x,s_2)\rt|+ \lt|\WI_{\alpha}(x,s_1)\rt|\lt|\sin(s_1)-\sin(s_2)\rt|\nn\\
	 	&\qd\qd\overset{(\ref{cpeq14.6}), (\ref{cpeq14.4})}{\leq} C\lt|s_1-s_2\rt|^{\alpha} \nu(x)\qd\text{ for all bounded }s_1, s_2\text{ and a.e. }x\in\Omega.
	 \end{align*}
	 By $2\pi$-periodicity of $\sin(\cdot)\WI_{\alpha}(x,\cdot)$, we may assume without loss of generality that $|\theta^h(x)-\theta^{-h}(x)|\leq \pi$. It follows from the above estimate and \eqref{eqaux11} that 
	 \begin{align}
	 	\label{cpeq264}
	 	\lt|\wt{\Xi}_2(h)\rt|&\leq  C\int_{\Omega} \nu(x) \lt|m(x-he_1)-m(x+he_1)\rt|^{\alpha}\gamma(x) \,dx\nn\\
	 	&\leq C\int_{\Omega}\nu(x)\lt(\lt|D^{-h}m(x)\rt|^\alpha+\lt|D^h m(x)\rt|^\alpha\rt)\gamma(x)\,dx.
	 \end{align}
	 Putting (\ref{cpeq264}) together with (\ref{cpeq173.4}) establishes (\ref{cpeq269}). Finally, putting \eqref{cpeq201.4} and \eqref{cpeq269} together with \eqref{cpeq201.5} establishes \eqref{cpeq283}.
\end{proof}

\begin{lem}
	\label{lem:A_alpha}
	Let $A_{\alpha}(x,h)$ be given in \eqref{eq:A_alpha}. For domains $U\subset\subset\Omega'\subset\subset\Omega''\subset\subset\Omega$ and any $\gamma\in C^{\infty}_c(\Omega')$, the estimate \eqref{estA_alpha} holds for all $|h|<\mathrm{dist}\{\Omega',\partial\Omega''\}$.
\end{lem}

\begin{proof}
	First we need to change $A_{\alpha,2}$ into a more convenient form to get estimates.  Since $\varphi_{\alpha}$ is odd, we have
	\begin{align*}
		&\iint\varphi_{\alpha}(s-t)\sin(s) \sin(t) \,\chi(x,s) \chi(x,t) \, dt\, ds\nn\\
		&\qd\qd\overset{\ti{s}=t, \ti{t}=s}{=}\iint\varphi_{\alpha}\lt(\ti{t}-\ti{s}\rt)\sin\lt(\ti{s}\rt) \sin\lt(\ti{t}\rt) \,\chi(x,\ti{s}) \chi(x,\ti{t}) \, d\ti{t}\, 
		d\ti{s}\nn\\
		&\qd\qd =-\iint\varphi_{\alpha}\lt(\ti{s}-\ti{t}\rt)\sin\lt(\ti{s}\rt) \sin\lt(\ti{t}\rt) \,\chi(x,\ti{s}) \chi(x,\ti{t}) \, d\ti{t}\, 
		d\ti{s},
	\end{align*}
	and thus $\iint\varphi_{\alpha}(s-t)\sin(s) \sin(t) \,\chi(x,s) \chi(x,t) \, dt\, ds=0$. So we have
	\begin{align*}
		A_{\alpha,2}(x,h)\overset{(\ref{eq:A_alpha})}{=}2 \iint\varphi_{\alpha}(s-t)\sin(s) \sin(t) \,\chi(x,s) D^h\chi(x,t) \, dt\, ds.
	\end{align*}
	By $\pi$-periodicity of $\varphi_{\alpha}$, it is easy to check that
	$f^{x,h}(t):=\int_{\theta^h(x)-\frac{\pi}{2}}^{\theta^h(x)+\frac{\pi}{2}} \varphi_{\alpha}(s-t)\sin(s)\cos(t)\, ds$ is $2\pi$-periodic with zero average on $[0,2\pi]$.
	Further it is clear that  $\|f^{x,h}\|_{L^{\infty}(\R/2\pi\mathbb{Z})}\lesssim \|\varphi_{\alpha}\|_{C^0}$. Using Lemma 
	\ref{auxteclem1} we deduce that
	\begin{align*}
		\lt|A_{\alpha,1}(x,h)\rt|\overset{(\ref{eq:A_alpha})}{=}2\lt|\int_0^{2\pi} f^{x,h}(t) D^h \chi(x,t)\, dt\rt|\overset{(\ref{eqaux1})}{\leq} C\lt|D^hm(x)\rt|\qd\text{ for a.e. }x\in\Omega'.
	\end{align*}
	And in the same way we have
	\begin{equation*}
		\lt|A_{\alpha,2}(x,h)\rt|\leq C\lt|D^hm(x)\rt|\qd\text{ for a.e. }x\in\Omega'.
	\end{equation*}
	This establishes \eqref{estA_alpha}.
\end{proof}

%
%

%
%

%
%

\section{Proof of Theorem \ref{C1}}

%

%
%

%

%
%
%

%
%
%

In this section we give the proof of Theorem \ref{C1}. In Subsection \ref{sec:main} we first compute the entropy production $\dv\Phi(m)$ for a wide class of $\Phi\in ENT$ in terms of the two special entropy productions $\dv\Sigma_j(m)$; see \eqref{eqmma5} of Proposition \ref{L17} below. This constitutes the major step of the proof. Then in Subsection \ref{sec:harmext}, using the specific structure of the entropies $\Phi_f$, we realize that the general formula \eqref{eqmma5} turns into the more specific one \eqref{cpeq2} for this class of entropies, which in turn gives the explicit characterization of the kinetic measure $\sigma$ as in \eqref{abeqa1.5}.

\subsection{Factorization for general entropies}
\label{sec:main}

In this subsection, we establish in Proposition \ref{L17} a more general version of the formula \eqref{cpeq2} for sufficiently regular entropies $\Phi\in ENT$.  Let $E:L^2(\mathbb{S}^1)\rightarrow L^2(B_1)$ be the continuous linear operator uniquely determined by its action on Fourier modes
\begin{equation}
	\label{eqhex1}
	E\psi_k(r e^{i\theta})=r^{\lt|k\rt|}e^{i k\theta}\qd\text{ for }\psi_k\lt(e^{i\theta}\rt)=e^{ik\theta}, k\in \mathbb{Z}.
\end{equation}
\begin{rem}
		\label{teclem}
		The operator $E$ is a continuous operator from $ C^4(\mathbb{S}^1)$ to  $C^3(\overline B_1)$. This is a well known fact and is also shown at the beginning of the proof of \cite[Lemma 16]{llp}. We include an argument for the convenience of the reader. From (\ref{eqhex1}) we have that $\|E\psi_k\|_{C^3(\overline B_1)}\lesssim 1+|k|^3$. Given
		$\psi=\sum_{k}c_k(\psi)\psi_k \in C^4(\mathbb{S}^1)$,  by Parseval's identity, we have $\sum_{k}\lt(k^4|c_k(\psi)|\rt)^2=\sum_k\lt|c_k\lt(\psi^{(4)}\rt)\rt|^2 = \frac{1}{2\pi}\lVert\psi^{(4)}\rVert_{L^2(\mathbb{S}^1)}^2\leq \lVert\psi^{(4)}\rVert_{C^0(\mathbb{S}^1)}^2$. It follows that
		\begin{equation*}
			\sum_{|k|\geq 1} \lt|c_k(\psi)\rt|\lt\lVert E\psi_k\rt\rVert_{C^3(\overline B_1)}\lesssim\sum_{|k|\geq 1} k^4 \lt|c_k(\psi)\rt|\cdot |k|^{-1}\leq C\lt\lVert\psi^{(4)}\rt\rVert_{C^0(\mathbb{S}^1)},
		\end{equation*}
		from which we deduce $\lt\lVert E\psi\rt\rVert_{C^3(\overline B_1)}\leq C\lVert\psi\rVert_{C^4(\mathbb{S}^1)}$.
	\end{rem}

As the linear span of $\{\psi_k\}_{k\in\mathbb{Z}}$ is dense in $C^4(\mathbb{S}^1)$ and $\Delta(E\psi_k)=0$ for all $k\in \mathbb{Z}$, it follows from 
	Remark \ref{teclem} that
\begin{equation*}
	\Delta(E\psi)=0\qd\text{ in }B_1\text{ for all }\psi\in C^4(\mathbb{S}^1).
\end{equation*}
In the following we refer to $E$ as the \emph{harmonic extension operator}.
\begin{prop}
	\label{prop:decomposition}
	Let $m$ satisfy \eqref{ageq3}. Assume  \eqref{cpeq1.23} holds true for some sequence $\ep_k\rightarrow 0$ and $1\leq p<\infty$. Given any $\psi\in C^4(\mathbb{S}^1)$ and $\varphi=E\psi$ its harmonic extension to $\overline{B}_1$, the harmonic entropy given by 
	\begin{equation}
		\label{harment}
		\Phi^{E\psi}(z)=\Phi^{\varphi}(z)=\varphi(z)z+\lt(iz\cdot\na\varphi(z)\rt)iz\qd\text{ for all } z\in\overline{B}_1
	\end{equation}
	satisfies $\dv\Phi^{E\psi}(m)\in L^{\blue{p}}_{\loc}(\Omega)$. Further
	it explicitly holds
	\begin{equation}
		\label{Entdecom}
		\dv\Phi^{E\psi}(m)=\AI_1\psi(m) e^{i2\theta}\cdot\dv\Sigma(m)+ \AI_2\psi(m) i e^{i2\theta}\cdot\dv\Sigma(m) \qd\text{ a.e. in }\Omega,
	\end{equation}
	where $\theta:\Omega\to[0,2\pi)$ satisfies $m(x)=e^{i\theta(x)}$ a.e. in $\Omega$, and $\AI_1, \AI_2: C^4(\mathbb{S}^1)\to C^0(\mathbb{S}^1)$ are the Fourier multiplier operators given by 
	\begin{equation*}
		\AI_1\psi_k=\frac{i k}{2} (k^2-1)\psi_k
	\end{equation*}
	and 
	\begin{equation*}
		\AI_2\psi_k=-\frac{\lt|k\rt|}{2}(k^2-1)\psi_k
	\end{equation*}
	for $\psi_k(e^{i\theta})=e^{ik\theta}$ for all $k\in\mathbb{Z}$.
\end{prop}

\begin{rem}
	The Fourier multiplier operators $\AI_1, \AI_2$ are defined for complex-valued functions on $\mathbb{S}^1$, but in \eqref{Entdecom} we only use their restrictions to real-valued functions. In the sequel we often implicitly make the identification $\R^2\cong \mathbb{C}$. 
\end{rem}

Proposition \ref{prop:decomposition} follows directly from the following Lemmas \ref{lemp5} and \ref{fmlem}. First we need a technical lemma.

\begin{lem} 
	\label{lembite} 
	Let $m$ satisfy \eqref{ageq3} and denote by $m_{\ep}=m\ast\rho_{\ep}$. Assume  \eqref{cpeq1.23} holds true for some sequence $\ep_k\rightarrow 0$ and $1\leq p<\infty$. Then for any $\MI\in C^0\lt(\overline B_1\rt)$, upon extraction of a subsequence, we have
	\begin{equation}
		\label{eqfv30}
		\lim_{k\to \infty}\int_{\Omega}\MI\lt(m_{\ep_k}\rt)\dv \Sigma_j\lt(m_{\ep_k}\rt)\zeta\,dx=\int_{\Omega}\MI\lt(m\rt)\dv \Sigma_j\lt(m\rt)\zeta\,dx \qd\text{ for all }\zeta\in C^\infty_c(\Omega),  j=1,2,
	\end{equation}
	where $\Sigma_1$ and $\Sigma_2$ are given in \eqref{ep101} and \eqref{ep102}, respectively. 
\end{lem}
\begin{proof}  From equations (25) and (26) in \cite{llp} we have
	\begin{equation*}
		\dv \Sigma_1\lt(m_{\ep}\rt)=\lt(\partial_1 m_{\ep 2}+\partial_2 m_{\ep 1}  \rt)\lt(1-\lt|m_{\ep}\rt|^2\rt),\qd \dv \Sigma_2\lt(m_{\ep}\rt)=\lt(\partial_2 m_{\ep 2}-\partial_1 m_{\ep 1}  \rt)\lt(1-\lt|m_{\ep}\rt|^2\rt).
	\end{equation*}
	It follows from \cite[Lemma 9]{llp} that
	$\lt|\dv \Sigma_j\lt(m_{\ep}\rt)(x)\rt|\lesssim \PPI^{\ep}_m(x)$ for all $x\in \Omega$, $0<\ep<\mathrm{dist}\{x, \partial \Omega\}$ and $j=1,2$. Given $U\subset\subset\Omega$, by assumption \eqref{cpeq1.23}, we know in particular that $\{\PPI^{\ep_k}_m\}\subset L^1(U)$ is a bounded sequence that converges weakly to $\PPI_m\in L^1(U)$. By the Dunford-Pettis Theorem (see \cite[Theorem 1.38]{ambrosio}), the sequence $\{\PPI^{\ep_k}_m\}$ is equiintegrable in $U$. By  \cite[Proposition 1.27]{ambrosio}, it is clear that $\{\dv\Sigma_j(m_{\ep_k})\}\subset L^1(U)$ is also bounded and equiintegrable in $U$, and thus by the Dunford-Pettis Theorem again, upon extraction of a subsequence which is not relabeled, it converges weakly to some $W_j\in L^1(U)$. We claim that $W_j=\dv\Sigma_j(m)$. Indeed, by Lemma \ref{p:controlent} and Remark \ref{R13}, we know $\dv\Sigma_j(m)\in L^p_{\loc}(\Omega)$. Further, we have
	\begin{align*}
		-\int_{U}\Sigma_j\lt(m\rt)\cdot \na\zeta \, dx&=\lim_{k\rightarrow \infty} -\int_{\Omega} \Sigma_j\lt(m_{\ep_k}\rt)\cdot\na\zeta \, dx\\
		&=\lim_{k\to\infty}\int_{\Omega}\dv\Sigma_j(m_{\ep_k})\zeta\,dx= \int_{\Omega} W_j\,\zeta\, dx\qd\text{ for all }\zeta\in C^{\infty}_{c}(U), 
	\end{align*}
	which shows $W_j=\dv \Sigma_j(m)$, $j=1, 2$. Therefore we deduce that
	\begin{equation}
		\label{Sigmaconverge}
		\dv \Sigma_j\lt(m_{\ep_k}\rt)\rightharpoonup\dv \Sigma_j\lt(m\rt) \qd\text{ in }L^1_{\loc}(\Omega)\text{ for }j=1,2. 
	\end{equation}

	Next we establish \eqref{eqfv30}. As $\MI$ is continuous and, upon extraction of another subsequence (not relabeled), $m_{\ep_k}\to m$ a.e., it follows that $\MI\lt(m_{\ep_k}\rt)\to \MI(m)$ a.e. in $\Omega$. Given $U\subset\subset\Omega$, by Egorov's theorem there exists a sequence of 
	subsets $U_l\subset U$ such that $\lt|U\setminus U_l\rt|\to 0$ as $l\to\infty$ and
	\begin{equation}
		\label{eqfv11.3}
		\MI\lt(m_{\ep_k}\rt) \to \MI(m)\qd\text{ in }L^{\infty}(U_l)\text{ as }k\rightarrow \infty\text{ for all }l.
	\end{equation}
	So for fixed $l$, it follows from \eqref{Sigmaconverge} and \eqref{eqfv11.3} that
	\begin{equation}
		\label{eqfv11}
		\lim_{k\rightarrow \infty}\int_{U_l}  \MI\lt(m_{\ep_k}\rt)  \dv \Sigma_j\lt(m_{\ep_k}\rt) \zeta\, dx=\int_{U_l}   \MI\lt(m\rt)  \dv \Sigma_j\lt(m\rt)  \zeta \, dx\qd\text{ for all }\zeta\in C^{\infty}_c(U).
	\end{equation}
	Now we have
	\begin{align}
		\label{eqfv12}
		&\lt|\int_{U}\MI\lt(m_{\ep_k}\rt)  \dv \Sigma_j\lt(m_{\ep_k}\rt) \zeta\, dx-\int_{U}\MI\lt(m\rt)  \dv \Sigma_j\lt(m\rt) \zeta\, dx\rt|\nn\\
		&\qd\leq \lt|\int_{U_l}\MI\lt(m_{\ep_k}\rt)  \dv \Sigma_j\lt(m_{\ep_k}\rt) \zeta\, dx-\int_{U_l}\MI\lt(m\rt)  \dv \Sigma_j\lt(m\rt) \zeta\, dx\rt|\nn\\
		&\qd\qd+\lt|\int_{U\setminus U_l}\MI\lt(m_{\ep_k}\rt)  \dv \Sigma_j\lt(m_{\ep_k}\rt) \zeta\, dx\rt|+\lt|\int_{U\setminus U_l}\MI\lt(m\rt)  \dv \Sigma_j\lt(m\rt) \zeta\, dx\rt|\nn\\
		&\qd\leq \lt|\int_{U_l}\MI\lt(m_{\ep_k}\rt)  \dv \Sigma_j\lt(m_{\ep_k}\rt) \zeta\, dx-\int_{U_l}\MI\lt(m\rt)  \dv \Sigma_j\lt(m\rt) \zeta\, dx\rt|\nn\\
		&\qd\qd+C\int_{U\setminus U_l}\lt|\dv \Sigma_j\lt(m_{\ep_k}\rt) \rt|\, dx+C\int_{U\setminus U_l}\lt|\dv \Sigma_j\lt(m\rt) \rt|\, dx. 
	\end{align}
	By equiintegrability of $\{\dv\Sigma_j(m_{\ep_k})\}$ and $\lt|U\setminus U_l\rt|\to 0$, we know
	\begin{equation}
		\label{dvSigmaremainder}
		\lim_{l\to \infty}\sup_k\int_{U\setminus U_l}\lt|\dv \Sigma_j\lt(m_{\ep_k}\rt)\rt| \, dx=0.
	\end{equation}
	Therefore, by first taking the limit $l\to\infty$ and then taking the limit $k\to\infty$ in \eqref{eqfv12}, we readily deduce \eqref{eqfv30} from \eqref{dvSigmaremainder} and \eqref{eqfv11}.
\end{proof}

\begin{lem}
	\label{lemp5} 
	Let $m$ satisfy \eqref{ageq3}. Assume  \eqref{cpeq1.23} holds true for some sequence $\ep_k\rightarrow 0$ and $1\leq p<\infty$. Given any $\varphi\in C^3(\overline{B}_1)$ satisfying $\Delta \varphi=0$, the harmonic entropy $\Phi^{\varphi}$ defined in \eqref{harment} satisfies $\dv\Phi^{\varphi}(m)\in L^{p}_{\loc}(\Omega)$. Further
	it explicitly holds
	\begin{equation}\label{ep112}
		\dv\Phi^{\varphi}(m)=A^{\varphi}_1(m) e^{i2\theta}\cdot\dv\Sigma(m)+ A^{\varphi}_2(m) i e^{i2\theta}\cdot\dv\Sigma(m) \qd\text{ a.e. in }\Omega,
	\end{equation}
	where $\theta:\Omega\to[0,2\pi)$ satisfies $m(x)=e^{i\theta(x)}$ a.e. in $\Omega$, and $A_1^{\varphi}, A_2^{\varphi}\in C^0(\overline B_1)$ are given by
	\begin{align}
		\label{ep113}
		A^{\varphi}_1(z)&:=\lt(z_1^2-z_2^2\rt)\lt(\frac{3}{2}\partial_{12}\varphi(z)-\frac{z_1}{2}\partial_{222}\varphi(z)+\frac{z_2}{2}\partial_{122}\varphi(z)\rt)\nn\\
		&\qd\qd +2 z_1 z_2\lt(-\frac{3}{2}\partial_{11}\varphi(z)+\frac{z_1}{2}\partial_{122}\varphi(z)-\frac{z_2}{2} \partial_{112}\varphi(z)  \rt)
	\end{align}
	and
	\begin{align}
		\label{ep114}
		A^{\varphi}_2(z)&:=-2 z_1 z_2\lt(\frac{3}{2}\partial_{12}\varphi(z)-\frac{z_1}{2}\partial_{222}\varphi(z)+\frac{z_2}{2}\partial_{122}\varphi(z)\rt)\nn\\
		&\qd\qd +\lt(z_1^2-z_2^2\rt)\lt(-\frac{3}{2}\partial_{11}\varphi(z)+\frac{z_1}{2}\partial_{122}\varphi(z)-\frac{z_2}{2} \partial_{112}\varphi(z)  \rt).
	\end{align}
\end{lem}

\begin{proof}
	Let $m_{\ep}=m\ast\rho_{\ep}$. By \cite[Lemma 18]{llp} we have that 
	\begin{equation}
		\label{eqfact1}
		\dv\Phi^{\varphi}(m_\ep)=\dv\lt(\lt(|m_\ep|^2-1\rt)B(m_\ep)\rt)+\partial_2 B_1(m_\ep)\dv\Sigma_1(m_\ep)-\partial_1 B_1(m_{\ep})\dv\Sigma_2(m_\ep),
	\end{equation}
	where $B=(B_1, B_2)$ and 
	\begin{align*}
		B_1(z)&=\partial_1\varphi(z) -\frac 12 z_1\partial_{22}\varphi(z) +\frac 12 z_2\partial_{12}\varphi(z),\\
		B_2(z)&=\partial_2\varphi(z) -\frac 12 z_2\partial_{11}\varphi(z) +\frac 12 z_1\partial_{12}\varphi(z).
	\end{align*}
	So, using $\Delta\varphi=0$, we have
	\begin{align*}
		\partial_1B_1(z) &= \partial_{11}\varphi(z)-\frac{1}{2}\partial_{22}\varphi(z)-\frac{1}{2}z_1\partial_{122}\varphi(z)+\frac{1}{2}z_2\partial_{112}\varphi(z)\\
		&=\frac{3}{2}\partial_{11}\varphi(z)-\frac{1}{2}z_1\partial_{122}\varphi(z)+\frac{1}{2}z_2\partial_{112}\varphi(z),\\
		\partial_2B_1(z) &= \partial_{12}\varphi(z)-\frac{1}{2}z_1 \partial_{222}\varphi(z)+\frac{1}{2}\partial_{12}\varphi(z)+\frac{1}{2}z_2 \partial_{122}\varphi(z)\\
		&=\frac{3}{2}\partial_{12}\varphi(z)-\frac{1}{2}z_1\partial_{222}\varphi(z)+\frac{1}{2}z_2\partial_{122}\varphi(z).
	\end{align*}
	Thus we can rewrite \eqref{eqfact1} as
	\begin{align}
		\label{xxeqa12}
		\dv\Phi^{\varphi}(m_\e) & =Q^{\varphi}_1(m_\e)\dv\Sigma_1(m_\e) +  Q^{\varphi}_2(m_\e)\dv\Sigma_2(m_\e)- \dv\lt(\lt(1-\abs{m_\e}^2\rt)B(m_\e)\rt),
	\end{align}
	where
	\begin{equation}
		\label{eqggbba4}
		Q^{\varphi}_1(z)=\frac{3}{2}\partial_{12}\varphi(z)-\frac{1}{2}z_1\partial_{222}\varphi(z)+\frac{1}{2}z_2\partial_{122}\varphi(z)
	\end{equation}
	and
	\begin{equation}
		\label{eqggbba5}
		Q^{\varphi}_2(z)=-\frac{3}{2}\partial_{11}\varphi(z)+\frac{1}{2}z_1\partial_{122}\varphi(z)-\frac{1}{2}z_2\partial_{112}\varphi(z).
	\end{equation}
	
	Hence for  any test function $\zeta\in C_c^1(\Omega)$ it holds
	\begin{align}\label{ep111}
		&-\int_{\Omega} \Phi^{\varphi}(m)\cdot\na\zeta\,dx = \lim_{\e\to 0} \int_{\Omega}\dv\Phi^{\varphi}(m_{\ep})\zeta\,dx\nn\\
		&\qd\qd\qd\qd\qd \overset{(\ref{xxeqa12})}{=} \lim_{\e\to 0} \int_{\Omega} (1-\abs{m_\e}^2)B(m_\e)\cdot\nabla\zeta\,dx + \lim_{\e\to 0} \sum_{j=1,2}\int_{\Omega} \dv \Sigma_j(m_\e)Q^{\varphi}_j(m_\e)\zeta\,dx.
	\end{align}
	The first term in the right-hand side is zero by the Dominated Convergence Theorem. From Lemma \ref{lembite} we also have 
	\begin{align*}
		\lim_{\ep\rightarrow 0}\int_{\Omega} \dv \Sigma_j(m_\e)Q^{\varphi}_j(m_\e)\zeta\,dx \overset{(\ref{eqfv30})}{=} \int_{\Omega}\dv\Sigma_j(m)Q^{\varphi}_j(m)\zeta\,dx\qd\text{ for }j=1,2.
	\end{align*}
	Plugging this into (\ref{ep111}) we obtain
	\begin{equation*}
		-\int_{\Omega} \Phi^{\varphi}(m)\cdot\na\zeta\,dx=\int_{\Omega}\dv\Sigma_1(m)Q^{\varphi}_1(m)\zeta\,dx+\int_{\Omega}\dv\Sigma_2(m)Q^{\varphi}_2(m)\zeta\,dx\qd\text{ for all }\zeta\in C^1_c(\Omega).
	\end{equation*}
	By Lemma  \ref{p:controlent} and Remark \ref{R13}, $\dv\Phi^{\varphi}(m), \dv\Sigma_j(m)\in L^{p}_{\loc}(\Omega)$. Thus we infer that
	\begin{equation}
		\label{eqggbba1}
		\dv\Phi^{\varphi}(m)=Q^{\varphi}_1(m)\dv\Sigma_1(m)+Q^{\varphi}_2(m)\dv\Sigma_2(m)\qd\text{ a.e. in }\Omega.
	\end{equation}
	Letting $Q^{\varphi}(m):=\lt(\begin{array}{c}  Q_1^{\varphi}(m)  \\   Q_2^{\varphi}(m)   \end{array}\rt)$ and $\RI(\theta):=\lt(\begin{array}{cc}  \cos(\theta)  & \sin(\theta) \\   -\sin(\theta) & \cos(\theta)   \end{array}\rt)$, it follows that
	\begin{align}
		\label{eqggbba2}
		\dv\Phi^{\varphi}(m)&\overset{(\ref{eqggbba1})}{=} Q^{\varphi}(m)\cdot \dv \Sigma(m)\nn\\
		&=\lt(\RI(2\theta) Q^{\varphi}(m)  \rt)\cdot \lt(\RI(2\theta)   \dv \Sigma(m) \rt)\nn\\
		&=\lt( \begin{array}{c} e^{i2\theta}\cdot Q^{\varphi}(m) \\   i e^{i2\theta}\cdot Q^{\varphi}(m) \end{array} \rt)\cdot    
		\lt( \begin{array}{c} e^{i2\theta}\cdot \dv \Sigma(m) \\   i e^{i2\theta}\cdot \dv \Sigma(m) \end{array} \rt)\nn\\
		&= \lt(e^{i2\theta}\cdot Q^{\varphi}(m) \rt) \lt( e^{i2\theta}\cdot \dv \Sigma(m)\rt)+ \lt(i e^{i2\theta}\cdot Q^{\varphi}(m) \rt) \lt(i e^{i2\theta}\cdot \dv \Sigma(m)\rt).
	\end{align}
	Note that
	\begin{equation}
		\label{eqggbba3}
		e^{i2\theta}\cdot Q^{\varphi}(m) =\lt(\begin{array}{c}  m_1^2-m_2^2 \\ 2 m_1 m_2\end{array}\rt)\cdot Q^{\varphi}(m),\qd
		ie^{i2\theta}\cdot Q^{\varphi}(m) =\lt(\begin{array}{c}  -2 m_1 m_2 \\ m_1^2-m_2^2 \end{array}\rt)\cdot Q^{\varphi}(m).
	\end{equation}
	Now putting (\ref{eqggbba4}), (\ref{eqggbba5}), (\ref{eqggbba2}) and (\ref{eqggbba3}) together gives (\ref{ep112})--(\ref{ep114}). 
\end{proof}

%
%

%
%

\begin{lem}
	\label{fmlem} 
	Define the operators $\AI_j: C^4(\mathbb{S}^1)\to C^0(\mathbb{S}^1), j=1, 2$ by
	\begin{equation}
		\label{eqhex2}
		\AI_j\psi:=A^{E\psi}_{j\, \lfloor \mathbb{S}^1}\qd\text{ for all }\psi\in C^4(\mathbb{S}^1),
	\end{equation}
	where $A_j$ are defined in \eqref{ep113}-\eqref{ep114}. Further let $\BI_1, \BI_2: C^4(\mathbb{S}^1)\to C^0(\mathbb{S}^1)$ be the Fourier multiplier operators characterized by
	\begin{equation}
		\label{eqggbba6}
		\BI_1\psi_k=\frac{i k}{2} (k^2-1)\psi_k\qd\text{ for all }  k\in \mathbb{Z}
	\end{equation}
	and 
	\begin{equation*}
		\BI_2\psi_k=-\frac{\lt|k\rt|}{2}(k^2-1)\psi_k\qd\text{ for all }  k\in \mathbb{Z},
	\end{equation*}
	where $\psi_k(e^{i\theta})=e^{ik\theta}$. Then we have
	\begin{equation}
		\label{AequalB}
		\AI_j\psi=\BI_j\psi\qd\text{ for } j=1, 2\text{ and for all }\psi\in C^4(\mathbb{S}^1).
	\end{equation}
\end{lem}

\begin{rem}
	As $\varphi\mapsto A^{\varphi}_j$ is linear, the definition of $\AI_j$ in \eqref{eqhex2} can be naturally extended to complex-valued $\psi$, and in particular $\psi_k$. We will make use of this implicitly in the sequel.
\end{rem}

\begin{proof}[Proof of Lemma \ref{fmlem}] 
	Using almost exactly the same arguments as in Remark \ref{teclem} one can show that $\AI_j, \BI_j$ are continuous operators from $C^4(\mathbb{S}^1)$ to $C^0(\mathbb{S}^1)$. Thus, to establish \eqref{AequalB}, it suffices to show $\AI_j\psi_k=\BI_j\psi_k$ for all $k\in\mathbb{Z}$ and $j=1, 2$ as the linear span of $\{\psi_k\}_{k\in\mathbb{Z}}$ is dense in $C^4(\mathbb{S}^1)$. 
	
	Let $\varphi_k=E\psi_k$. From \eqref{eqhex1} we have $\varphi_k(z)=z^k$ for $k\geq 0$, and thus 
	\begin{equation}
		\label{appeqd3}
		\partial_1\varphi_k=k z^{k-1}\qd\text{ and }\qd \partial_2\varphi_k=ik z^{k-1}. 
	\end{equation}
		So for $k\geq 3$ we have 
	\begin{align*}
		&\frac{3}{2}\partial_{12}\varphi_k-\frac{z_1}{2}\partial_{222}\varphi_k+\frac{z_2}{2}  \partial_{122} \varphi_k\nn\\
		&\qd=\frac{3}{2}ik(k-1)z^{k-2}+\frac{z_1}{2}ik(k-1)(k-2)z^{k-3}-\frac{z_2}{2} k (k-1)(k-2)z^{k-3}\nn\\
		&\qd=\frac{3}{2}ik(k-1)z^{k-2}+\frac{ik (k-1)(k-2)}{2}(z_1+iz_2)z^{k-3}\nn\\
		&\qd=\frac{ik(k-1)(k+1)}{2}z^{k-2}
	\end{align*}
	and 
	\begin{align*}
		&-\frac{3}{2}\partial_{11}\varphi_k+\frac{z_1}{2}\partial_{122}\varphi_k-\frac{z_2}{2}  \partial_{112} \varphi_k\nn\\
		&\qd=-\frac{3}{2}k(k-1)z^{k-2}-\frac{z_1}{2}k(k-1)(k-2)z^{k-3}-\frac{z_2}{2} i k (k-1)(k-2)z^{k-3}\nn\\
		&\qd=-\frac{3}{2}k(k-1)z^{k-2}-\frac{k(k-1)(k-2)}{2}(z_1+iz_2)z^{k-3}\nn\\
		&\qd=-\frac{k(k-1)(k+1)}{2}z^{k-2}.  
	\end{align*}
	Hence, it follows from \eqref{ep113}, \eqref{ep114} and the above computations that
	\begin{align*}
		A_1^{\varphi_k}&=\lt(z_1^2-z_2^2\rt) \frac{ik(k-1)(k+1)}{2}z^{k-2}-2 z_1 z_2\frac{ k(k-1)(k+1)}{2}z^{k-2}\nn\\
		&=\frac{i k(k^2-1)}{2}\lt(z_1^2-z_2^2 +2i z_1 z_2   \rt) z^{k-2}=\frac{i k(k^2-1)}{2} z^k,
	\end{align*}
	and
	\begin{align*}
		A_2^{\varphi_k}&=-2 z_1 z_2 \frac{ik(k-1)(k+1)}{2}z^{k-2}-\lt(z_1^2-z_2^2\rt)\frac{k(k-1)(k+1)}{2}z^{k-2}\nn\\
		&=-\frac{k(k^2-1)}{2}\lt(z_1^2-z_2^2 +2i z_1 z_2   \rt) z^{k-2}=-\frac{k(k^2-1)}{2} z^k.
	\end{align*}
	Thus we have
	\begin{equation*}
		\AI_1\psi_k=\frac{i k(k^2-1)}{2} \psi_k=\BI_1\psi_k,\qd \AI_2\psi_k=-\frac{k(k^2-1)}{2} \psi_k=\BI_2\psi_k\qd\text{ for all }k\geq 3.
	\end{equation*}
	For $0\leq k\leq 2$, $\AI_j\psi_k=\BI_j\psi_k$ can be checked directly. For $k<0$, note that $E\psi_k=\overline{E\psi_{-k}}$ and thus $A_j^{E\psi_k}=A_j^{\overline{E\psi_{-k}}}=\overline{A_j^{E\psi_{-k}}}$. Using this and the formulas established for $k>0$, one can check directly that $\AI_j\psi_k=\BI_j\psi_k$ for $k<0$. Thus we have established $\AI_j\psi_k=\BI_j\psi_k$ for all $k\in\mathbb{Z}$ and $j=1, 2$.
\end{proof}

Now combining Lemmas \ref{lemp5} and \ref{fmlem} completes the proof of Proposition \ref{prop:decomposition}. Next we show that the term $i e^{i2\theta}\cdot\dv\Sigma(m)$ must indeed vanish a.e. due to the specific structure of its coefficient $\AI_2\psi$, and thus the factorization formula \eqref{Entdecom} has only the first term on the right-hand side. This constitutes the most important step towards the proof of Theorem \ref{C1}.

%
%

%
%

\begin{prop} 
	\label{L17}
	Let $m$ satisfy \eqref{ageq3}. Assume  \eqref{cpeq1.23} holds true for some sequence $\ep_k\rightarrow 0$ and $1\leq p<\infty$. Given any $\psi\in C^4(\mathbb{S}^1)$, the harmonic entropy $\Phi^{E\psi}$ given by \eqref{harment} satisfies $\dv\Phi^{E\psi}(m)\in L^p_{\loc}(\Omega)$ and the formula
	\begin{equation}
		\label{eqmma5}
		\dv \Phi^{E \psi}(m)=\AI_1\psi(m) e^{i2\theta}\cdot \dv \Sigma(m)\qd\text{ a.e. in }\Omega,
	\end{equation}
	where $\theta:\Omega\to[0,2\pi)$ satisfies $m(x)=e^{i\theta(x)}$ a.e. in $\Omega$. Further, it explicitly holds
	\begin{equation}
		\label{eqmma6}
		\AI_1\psi=-\frac{1}{2}\lt(\psi^{(3)}+\psi'\rt)\qd\text{ for all }\psi\in C^4(\mathbb{S}).
	\end{equation}
\end{prop}
\begin{proof} 
	Let $\wt\BI_1: C^3(\mathbb{S}^1)\to C^0(\mathbb{S}^1)$ be the continuous linear operator defined by $\wt\BI_1\psi:=-\frac 1 2 \lt(\psi^{(3)}+\psi'\rt)$ for all $\psi\in C^3(\mathbb{S}^1)$. From \eqref{AequalB} and \eqref{eqggbba6}, it is clear that $\AI_1\psi_k=\BI_1\psi_k=\wt\BI_1\psi_k$ for all $k\in\mathbb{Z}$. As $\AI_1$ and $\wt\BI_1$ are both continuous from $C^4(\mathbb{S}^1)$ to $C^0(\mathbb{S}^1)$ and the linear span of $\{\psi_k\}_{k\in\mathbb{Z}}$ is dense in $C^4(\mathbb{S}^1)$, it follows that \eqref{eqmma6} holds for all $\psi\in C^4(\mathbb{S})$.
	
	Now comparing \eqref{eqmma5} with \eqref{Entdecom}, it suffices to show $i e^{i2\theta}\cdot \dv \Sigma\lt(m\rt)=0$ a.e. in $\Omega$. To this end, we argue in a very similar way to the proof of \cite[Theorem 1]{llp}. Let $\mathcal X \subset C^4(\mathbb{S}^1)$ denote a countable dense subset. Let $\mathcal G\subset\Omega$ be the set of all points $x\in\Omega$ at which  $\PPI_m(x)<\infty$ and $|\dv\Sigma_j(m)(x)|<\infty$ for $j=1, 2$, and both:
	\begin{itemize}
		\item the explicit expression of $\dv\Phi^{E\psi}(m)$ given by \eqref{Entdecom},
		\item its control in terms of $\norm{\Phi^{E\psi}}_{C^2(\mathbb{S}^1)}$ given by \eqref{eq621.2},
	\end{itemize}
	hold for all $\psi\in\mathcal X$. By Proposition~\ref{prop:decomposition} and Lemma~\ref{p:controlent}, the set $\mathcal G$ has full measure in $\Omega$. From \eqref{harment} we have $\Phi^{E\psi}(e^{i\theta})=\psi(e^{i\theta})e^{i\theta}+\psi'(e^{i\theta})ie^{i\theta}$, and thus $\|\Phi^{E\psi}\|_{C^2(\mathbb{S}^1)}\lesssim \|\psi\|_{C^3(\mathbb{S}^1)}$. Therefore the estimate \eqref{eq621.2} becomes
	\begin{equation}
		\label{eqiop11}
		\lt|\dv \Phi^{E\psi}(m)(x) \rt|\lesssim \|\psi\|_{C^3\lt(\mathbb{S}^1\rt)}\PPI_m(x)\qd\text{ for all }x\in\mathcal{G}, \psi\in\mathcal X.
	\end{equation}
	And we know from \eqref{eqmma6} that $\|\mathcal{A}_1\psi\|_{C^0\lt(\mathbb{S}^1\rt)}\lesssim \|\psi\|_{C^3\lt(\mathbb{S}^1\rt)}$ for all $\psi\in C^4(\mathbb{S}^1)$. Suppose for some $x_0\in \mathcal{G}$ we have $i e^{i2\theta(x_0)}\cdot \dv \Sigma\lt(m\rt)(x_0)\not =0$, then it follows from \eqref{Entdecom} and \eqref{eqiop11} that
	\begin{align}
		\label{eqiop12}
		\lt|\AI_2\psi(m(x_0))\rt|&\lesssim \lt|i e^{i2\theta(x_0)}\cdot \dv \Sigma\lt(m\rt)(x_0)\rt|^{-1}\lt(\PPI_m(x_0)+\lt|\dv\Sigma(m)(x_0)\rt|\rt)\|\psi\|_{C^3\lt(\mathbb{S}^1\rt)}\nn\\
		&=C(m,x_0) \|\psi\|_{C^3\lt(\mathbb{S}^1\rt)}\qd\text{ for all }\psi\in\mathcal X.
	\end{align}
	As $\mathcal X\subset C^4(\mathbb{S}^1)$ is dense in $C^4(\mathbb{S}^1)$ and both sides of \eqref{eqiop12} depend on $\psi$ continuously in the $C^4$ topology, we deduce that
	\begin{equation*}
		\lt|\AI_2\psi(m(x_0))\rt|\lesssim C(m,x_0) \|\psi\|_{C^3\lt(\mathbb{S}^1\rt)}\qd\text{ for all }\psi\in C^4(\mathbb{S}^1). 
	\end{equation*}
	
	The remainder of the argument follows very closely the proof of \cite[Theorem 1]{llp} and thus we sketch it somewhat briefly. We identify functions on $\mathbb{S}^1$
	with $2\pi$-periodic functions on $\mathbb{R}$, and
	$m(x_0)\in \mathbb{S}^1$ with its argument $\theta_0\in \mathbb{R}/ 2\pi \mathbb{Z}$. So we have $\lt|\AI_2\psi(\theta_0)\rt|\lesssim C(m, x_0)
	\|\psi\|_{C^3\lt(\mathbb{S}^1\rt)}$ for all $\psi\in C^4(\mathbb{S}^1)$. This estimate turns into the stronger estimate \begin{equation}
		\label{A2}
		\lVert \AI_2\psi\rVert_{L^{\infty}(\mathbb{S}^1)}\lesssim C(m,x_0)\|\psi\|_{C^3(\mathbb{S}^1)}\qd\text{ for all }\psi\in C^4(\mathbb{S}^1)
	\end{equation}	
	because the multiplier operator $\AI_2$ commutes with translations of the variable. Now decomposing $\AI_2=\AI_2^1+\AI_2^2$, where $\AI_2^1\psi_k=-\frac{|k|^3}{2}\psi_k$ and $\AI_2^2\psi_k=\frac{|k|}{2}\psi_k$ for all $k\in\mathbb{Z}$, a simple argument shows that $\AI_2^2$ can be extended to a continuous linear operator from $C^3(\mathbb{S}^1)$ to $C^0(\mathbb{S}^1)$. Thus the estimate  \eqref{A2} gives 
	\begin{equation}
		\label{contradiction}
		\|\AI_2^1\psi\|_{L^{\infty}(\mathbb{S}^1)}\lesssim \wt C \|\psi\|_{C^3(\mathbb{S}^1)}\qd\text{ for all }\psi\in C^4(\mathbb{S}^1). 
	\end{equation}
	This estimate is of exactly the same type as the estimate (51) in \cite{llp}, which is shown to lead to a contradiction because of well known properties of the Hilbert transform; see the arguments below equation (51) in \cite{llp} for the details. The  contradiction from the estimate \eqref{contradiction} shows that $i e^{i2\theta(x)}\cdot \dv \Sigma\lt(m\rt)(x)=0$ for all $x\in \mathcal G$, and thus establishes (\ref{eqmma5}).
\end{proof}

\subsection{Harmonic extensions for $\Phi_f$}
\label{sec:harmext}

In this subsection we  construct specific extensions to $\overline B_1$ of the entropies $\Phi_f$ defined by \eqref{eqplm1} on $\mathbb S^1$. For $f$ sufficiently regular, these extensions are harmonic entropies and enjoy the nice formula \eqref{eqmma5} established in Proposition \ref{L17}. This fact coupled with the special structure of the entropies $\Phi_f$ leads to the explicit structure of the kinetic measure $\sigma$ in \eqref{abeqa1.5}. 

Given $f\in L^2(\R/2\pi\Z)$, let $a_k(f)$ and $b_k(f)$ denote the standard Fourier coefficients of $f$, i.e.
\begin{equation*}
	a_k(f)=\frac{1}{\pi}\int_0^{2\pi} f(t)\cos(kt)\,dt,\qquad b_k(f)=\frac{1}{\pi}\int_0^{2\pi} f(t)\sin(kt)\,dt\qquad\text{for }k\geq 1.
\end{equation*}
We define $\xi_f\colon \overline B_1\to\R$ by
\begin{equation}\label{eq:xif}
	\xi_f = \sum_{k\geq 1} \frac{(-1)^k}{k(1-2k)(1+2k)}\left( - a_{2k}(f) \varphi^1_{2k} + b_{2k}(f)\varphi^2_{2k}\right),
\end{equation}
where $\varphi_k=\varphi_k^2+i\varphi_k^1$ and $\varphi_k(z)=z^k$ for $z\in\mathbb{C}$ and $k\geq 0$. So explicitly in polar coordinates we have
\begin{equation*}
	\varphi^1_k =r^k\sin(k\theta),\quad\varphi^2_k=r^k\cos (k\theta).
\end{equation*}
Note that this choice of indices for $\varphi_k^1$ and $\varphi_k^2$ is apparently natural from some of the following computations; see e.g. \eqref{eqad12} and \eqref{eqad13}.

\begin{prop}\label{p:harmext}
	For all $f\in L^2(\R/2\pi\Z)$, the function $\xi_f$ given by \eqref{eq:xif} belongs to $C^2(\overline B_1)$ and solves $\Delta\xi_f =0$ in $B_1$. The entropy given by
	\begin{equation}\label{eq:extPhif}
		\Phi^{\xi_f}(z)=\xi_f(z)z +((iz)\cdot\nabla\xi_f(z))iz\qquad\forall z\in\overline B_1
	\end{equation}
	extends the entropy $\Phi_f$ defined by \eqref{eqplm1} on $\mathbb S^1=\partial B_1$, up to a linear term:
	\begin{equation}\label{eq:harmext}
		\Phi^{\xi_f}(z)=\Phi_f(z)  +  \left(2\sum_{k\geq 2}\frac{b_{k}(f)}{k}\right) z\qquad\forall z\in\partial B_1.
	\end{equation}
	If moreover $f\in H^\ell(\R/2\pi\Z)$ for some $\ell\in\mathbb N$, then the series \eqref{eq:xif} defining $\xi_f$ converges in $C^{\ell+2}(\overline B_1)$.
\end{prop}

The proof of Proposition~\ref{p:harmext} follows from direct calculations showing its validity for the Fourier modes $f(t)=\cos(kt),\sin(kt)$, and from standard estimates on Fourier coefficients ensuring that the claimed convergence and regularity hold.

\begin{proof}[Proof of Proposition~\ref{p:harmext}]
	First notice that since 
	$z^k=\varphi^2_k(z)+i\varphi^1_k(z)$ for $z\in\mathbb{C}$, we have
	\begin{equation}
		\label{xxeqa1}
		\|\nabla^\ell\varphi^j_{k}\|_{C^{0}\lt(\overline{B}_1\rt)}\lesssim k^\ell \qquad\forall k\geq 1, 
	\end{equation}
	and since by Parseval's identity the Fourier coefficients $a_{k}(f)$, $b_k(f)$ belong to $\ell^2$, we have
	\begin{equation*}
		\left\Vert \frac{- a_{2k}(f) \varphi^1_{2k} + b_{2k}(f)\varphi^2_{2k}}{k(1-2k)(1+2k)}\right\Vert_{C^2(\overline B_1)}\overset{\eqref{xxeqa1}}{\lesssim} \frac 1 k (\abs{a_{2k}(f)}+\abs{b_{2k}(f)})\in\ell^1.
	\end{equation*}
	Hence the series \eqref{eq:xif} converges in $C^2(\overline B_1)$, and using H\"{o}lder's inequality for the following first inequality we obtain
	\begin{equation}
		\label{iteq1}
		\norm{\xi_f}_{C^2(\overline B_1)}\lesssim \norm{(a_k(f))}_{\ell^2} + \norm{(b_k(f))}_{\ell^2}\lesssim \norm{f}_{L^2(\R/2\pi\Z)}.
	\end{equation}
	As all terms of the series \eqref{eq:xif} are harmonic functions,  it follows that $\Delta\xi_f=0$ in $B_1$. 
	
	If $f\in H^\ell(\R/2\pi\Z)$ then the sequences $(k^\ell a_k(f))$, $(k^\ell b_k(f))$ belong to $\ell^2$
	since they are, up to constants, the Fourier coefficients of $f^{(\ell)}$. Thus
	\begin{equation}
		\label{iteq2}
		\left\Vert \frac{- a_{2k}(f) \varphi^1_{2k} + b_{2k}(f)\varphi^2_{2k}}{k(1-2k)(1+2k)}\right\Vert_{C^{\ell+2}(\overline B_1)}
		\overset{\eqref{xxeqa1}}{\lesssim} \frac 1 k (k^\ell\abs{a_{2k}(f)}+k^\ell\abs{b_{2k}(f)})\lesssim \lVert f^{(\ell)}\rVert_{L^2(\mathbb{R}/2\pi\mathbb{Z})},
	\end{equation}
	and the series \eqref{eq:xif} converges in $C^{\ell+2}(\overline B_1)$.
	
	It remains to prove that the harmonic entropy $\Phi^{\xi_f}$ indeed extends $\Phi_f$, i.e. we have \eqref{eq:harmext}. 
	Since we have just shown in \eqref{iteq1} that the linear map $f\mapsto\xi_f$ is continuous $L^2(\R/2\pi\Z)\to C^2(\overline B_1)$, we have in particular (recall \eqref{eq:extPhif}) that for any $z\in \mathbb S^1$ the linear map $f\mapsto\Phi^{\xi_f}(z)$ is continuous in $L^2(\R/2\pi\Z)$. By the construction of $\Phi_f$ (recalling (\ref{eqplm1}), (\ref{eq:varphi_f}), (\ref{eqplm1.5})) the two terms in the right-hand side of \eqref{eq:harmext} depend also linearly and continuously on $f\in L^2(\R/2\pi\Z)$. 
	Therefore it is sufficient to establish \eqref{eq:harmext} for $f(t)=\cos(kt),\sin(kt)$ for all $k\geq 0$. For $k\geq 2$ this follows from lengthy but direct computations, to be found in Appendix~\ref{a:compharmext}. And for $k\in\lbrace 0,1\rbrace$ it can be checked directly that both sides of \eqref{eq:harmext} vanish.
\end{proof}

%
%

\begin{lem}
	\label{LLAA8}
	Let $f\in C^2\lt(\mathbb{R}/ 2\pi \mathbb{Z}\rt)$ and $\xi_f$ be the function defined by \eqref{eq:xif}. Then
	\begin{equation}
		\label{eqoolla1}
		\AI_1\xi_{f \lfloor \mathbb{S}^1}(e^{i\theta})=\frac 12 f\lt(\theta+\frac{\pi}{2}\rt)+\frac 12 f\lt(\theta-\frac{\pi}{2}\rt)-\langle f, 1\rangle, 
	\end{equation}
	where recall that $\langle \cdot, \cdot \rangle$ is the inner product on $L^2\lt(\mathbb{R}/ 2\pi \mathbb{Z}\rt)$, i.e. $\langle f, g\rangle = \frac{1}{2\pi}\int_0^{2\pi} f\, g\, dt$.
\end{lem}
\begin{proof}  
	Given $f\in C^2(\R/2\pi\mathbb{Z})$, it follows from the continuity of $\AI_1$ from $C^4(\mathbb{S}^1)$ to $C^0(\mathbb{S}^1)$ (recalling Lemma \ref{fmlem}) and the estimate \eqref{iteq2} that 
	\begin{equation*}
	\lVert \AI_1\xi_{f \lfloor \mathbb{S}^1}\rVert_{C^0(\mathbb{S}^1)}\lesssim  \lVert \xi_{f \lfloor \mathbb{S}^1}\rVert_{C^4(\mathbb{S}^1)}\overset{(\ref{iteq2})}{\lesssim} \lVert f\rVert_{C^2(\mathbb{R}/2\pi\mathbb{Z})}.
	\end{equation*}
	It is clear that the right-hand side of \eqref{eqoolla1} also depends continuously on $f$ in the $C^2$ topology, and both sides of \eqref{eqoolla1} are linear in $f$. Thus it suffices to show \eqref{eqoolla1} for $f=f^j_k$, $j=1,2$, $k\geq 0$, where
	\begin{equation}
		\label{f_k}
		f^1_k(t)=\cos(kt),\quad f^2_k(t)=\sin(kt),
	\end{equation}
	as the linear span of $\{f^j_k\}$ is dense in $C^2(\R/2\pi\mathbb{Z})$.
	
	Note that since $\psi_{2k}=\varphi_{2k}^2+i \varphi_{2k}^1$ and
	\begin{equation*}
		\AI_1\psi_{2k}\overset{ (\ref{eqggbba6}), (\ref{AequalB})}{=}\frac{i2k (4k^2-1)}{2}\psi_{2k}=  k(4k^2-1)\lt(-\varphi_{2k}^1+i\varphi_{2k}^2 \rt),
	\end{equation*}
	we have
	\begin{equation}
		\label{eqoolla2}
		\AI_1\varphi_{2k}^2=-k (4k^2-1)\varphi_{2k}^1
		\qd\text{ and }\qd\AI_1\varphi_{2k}^1=k(4k^2-1)\varphi_{2k}^2.
	\end{equation}
	For $f=f_{2k+1}^j(t)$, $j=1,2$, $k\geq 0$ and $f\equiv 1$,  it is clear that $\xi_f\overset{(\ref{eq:xif})}{=}0$, and thus both sides of \eqref{eqoolla1} vanish.
	From \eqref{eq:xif}, for $k\geq 1$, we have $\xi_{f_{2k}^1}=\frac{(-1)^{k+1}}{k(1-4k^2)}\varphi^1_{2k}$ and $\xi_{f_{2k}^2}=\frac{(-1)^{k}}{k(1-4k^2)}\varphi^2_{2k}$. Thus it follows from \eqref{eqoolla2} that
	\begin{align*}
		\AI_1\xi_{f_{2k}^1 \lfloor \mathbb{S}^1}(e^{i\theta})&=\frac{(-1)^{k+1}}{k(1-4k^2)}\cdot k(4k^2-1)\varphi^2_{2k}(e^{i\theta})=(-1)^k\cos(2k\theta)\\
		&=\frac 12 f_{2k}^1\lt(\theta+\frac{\pi}{2}\rt)+\frac 12 f_{2k}^1\lt(\theta-\frac{\pi}{2}\rt), 
	\end{align*} 
	and 
	\begin{align*}
		\AI_1\xi_{f_{2k}^2 \lfloor \mathbb{S}^1}(e^{i\theta})&=\frac{(-1)^{k+1}}{k(1-4k^2)}\cdot k(4k^2-1)\varphi^1_{2k}(e^{i\theta})=(-1)^k\sin(2k\theta)\\
		&=\frac 12 f_{2k}^2\lt(\theta+\frac{\pi}{2}\rt)+\frac 12 f_{2k}^2\lt(\theta-\frac{\pi}{2}\rt). 
	\end{align*} 
	Thus we have established \eqref{eqoolla1} for all Fourier modes and this concludes the proof of the lemma.
\end{proof}

\subsection{Proof of Theorem \ref{C1} completed}

\begin{proof}[Proof of Theorem \ref{C1}]
	First note that by Lemma \ref{p:controlent} and Remark \ref{R13} we have $\dv\Phi_f(m), \dv\Sigma_j(m)\in L^p_{\loc}(\Omega)$. For any $f\in C^2\lt(\mathbb{R}/ 2\pi \mathbb{Z}\rt)$, as harmonic extension is unique, we deduce from Propositions \ref{L17}, \ref{p:harmext} and Lemma \ref{LLAA8} that 
\begin{align*}
	\dv \Phi_f(m)(x)&\overset{(\ref{eq:harmext})}{=}\dv \Phi^{\xi_f}(m)(x)\nn\\
	&\overset{(\ref{eqmma5}), (\ref{eqoolla1})}{=}\lt(\frac 12 f\lt(\theta(x)+\frac{\pi}{2}\rt)+\frac 12 f\lt(\theta(x)-\frac{\pi}{2}\rt)-\langle f, 1\rangle\rt)e^{2i\theta(x)}\cdot \dv \Sigma(m)(x)
\end{align*}
for a.e. $x\in\Omega$. Integrating by parts gives
\begin{align}
	\label{eqokl30}
	&-\int_{\Omega} \Phi_f(m)(x)\cdot \na \zeta(x) \, dx\nn\\
	&\qd\qd=\int_{\Omega} \lt(\frac 12 f\lt(\theta(x)+\frac{\pi}{2}\rt)+\frac 12 f\lt(\theta(x)-\frac{\pi}{2}\rt)-\langle f, 1\rangle\rt)e^{2i\theta(x)}\cdot \dv \Sigma(m)(x)\, \zeta(x) \, dx
\end{align}
for all $f\in C^2\lt(\mathbb{R}/ 2\pi \mathbb{Z}\rt)$ and all $\zeta\in C^\infty_c(\Omega)$. For any fixed $\zeta\in C^\infty_c(\Omega)$, the left- and right-hand sides of (\ref{eqokl30}) depend continuously and linearly on $f$ in the $C^0$ topology (recalling again the construction of $\Phi_f$ from (\ref{eqplm1}), (\ref{eq:varphi_f}), (\ref{eqplm1.5})). By density 
of $C^2\lt(\mathbb{R}/ 2\pi \mathbb{Z}\rt)$ in $C^0\lt(\mathbb{R}/ 2\pi \mathbb{Z}\rt)$
it follows that (\ref{eqokl30}) holds true for all $f\in C^0\lt(\mathbb{R}/ 2\pi\mathbb{Z}\rt)$ and all $\zeta\in C_c^{\infty}(\Omega)$. As $\dv\Phi_f(m), \dv\Sigma(m)\in L^p_{\loc}(\Omega)$, we readily deduce \eqref{cpeq2} from \eqref{eqokl30}. 

Finally, to establish \eqref{abeqa1.5}, note that the existence of $\sigma$ and its disintegration $\sigma=\mathcal{L}^2\otimes \sigma_x$ follow from Lemma \ref{lem:sigmaLp} in the case $1<p<\infty$. When $p=1$, by Lemma \ref{p:controlent}, the estimate \eqref{eq621.2} holds. Given $A\subset \subset \Omega$, for any pairwise disjoint decomposition $A=\bigcup_{\alpha} A_{\alpha}$ and any choice of 
	$\{\Phi_{\alpha}\}\subset ENT$ with $\|\Phi_{\alpha}\|_{C^2(\mathbb{S}^1)}\leq 1$, it follows from \eqref{eq621.2} that 
	$\sum_{\alpha} \int_{A_{\alpha}}\lt|\dv \Phi_{\alpha}(m)\rt|dx\lesssim \int_{A} \PPI_m \, dx$. Thus it follows from \eqref{eqinta2} that $\mu_m(A)\leq \int_{A} \PPI_m\, dx$ for all $A\subset \subset \Omega$.  This implies that $\mu_m$ is absolutely continuous with respect to the Lebesgue measure and its density is bounded above by a constant multiple of $\PPI_m\in L^1_{\loc}(\Omega)$, and thus $\mu_m\in L^1_{\loc}(\Omega)$. The existence of $\sigma$ and its disintegration $\sigma=\mathcal{L}^2\otimes \sigma_x$ then follow from Lemmas \ref{p:sigmaLp} and \ref{slice}. Now putting \eqref{eqgl1.2} and \eqref{eq30.2} together gives the equivalence of \eqref{cpeq2} and \eqref{abeqa1.5}.
\end{proof}

\begin{lem} 
	\label{slice}
	Let $m:\Omega\to\R^2$ satisfy \eqref{ageq3}. Assume $\mu_m\in L^p_{\loc}(\Omega)$ for some $1\leq p<\infty$, where $\mu_m$ is given in \eqref{eqinta2}. Then there exists a family of measures $\lt\{\sigma_x\rt\}\subset \mathcal{M}\lt(\mathbb{R}/ 2\pi \mathbb{Z}\rt)$ for $\mathcal{L}^2$-a.e. $x\in\Omega$ satisfying \eqref{eq30.2} and \eqref{eq:sigmaLp}.
\end{lem}

The above result is a fairly standard application of disintegration/slicing for the measure $\sigma$ constructed in Lemma \ref{p:sigmaLp} (or Lemma \ref{lem:sigmaLp} when $p>1$ and in this case the disintegration of $\sigma$ is already included in Lemma \ref{lem:sigmaLp}). By showing that the push-forward of $\sigma$ into $\Omega$ via the projection onto $\Omega$ is absolutely continuous under the assumption $\mu_m\in L^p_{\loc}$, it could  also be 
deduced from \cite[Theorem 2.28]{ambrosio}. However since establishing this is almost as involved as a direct proof we choose the latter and include the proof in Appendix \ref{a:sigmaLp}.

\section{Proof of Theorem \ref{C0}} 

The proof of Theorem \ref{C0} under the assumption \eqref{eqhyp1} is a direct application of Theorem \ref{C2}. Under the assumption \eqref{eqhyp2}, on the other hand, the proof makes use of the following

\begin{prop}
	\label{LLD1}
	Let $m$ satisfy \eqref{ageq3} and \eqref{eqhyp2}. Then $m\in B^{\frac{1}{3}}_{4,\infty,\loc}(\Omega)$.
\end{prop}
\begin{proof}
	The proof relies on the div-curl inequality
	\begin{equation}\label{eq:divcurl2}
		\abs{\int E\wedge B}\lesssim pp'\lt( \norm{E}_{L^{p'}}\norm{\dv B}_{W^{-1,p}}+\norm{B}_{L^{p'}}\norm{\dv E}_{W^{-1,p}}\rt),
	\end{equation}
	valid for all $p\in (1,\infty)$ and compactly supported bounded vector fields $E,B\colon\R^2\to\R^2$; see \cite[Lemma 4.2]{GL}. Given $U\subset\subset\Omega'\subset\subset\Omega''\subset\subset\Omega$ and $h\in\R^2$ with $|h|<\mathrm{dist}\{\Omega',\partial\Omega''\}$, we apply the estimate \eqref{eq:divcurl2} to
	\begin{equation*}
		E=\chi D^h\Sigma_1(m),\quad B=\chi D^h\Sigma_2(m),
	\end{equation*}
	where $\chi\in C^{\infty}_c(\Omega')$ is a cut-off function with $0\leq \chi\leq 1$ and $\chi\equiv 1$ on $\overline U$. First by \cite[Lemma 7]{LP} and noting that 
	$m_1=\Sigma_1\lt(m\rt)\cdot e_2- \Sigma_2\lt(m\rt)\cdot e_1$ and 
	$m_2=\Sigma_1\lt(m\rt)\cdot e_1+ \Sigma_2\lt(m\rt)\cdot e_2$, we have 
	\begin{equation*}
		E\wedge B \gtrsim  \chi^2 \abs{D^h \Sigma }^4    \gtrsim \chi^2 \abs{D^h m}^4,
	\end{equation*}
	where recall that $\Sigma=\lt(\Sigma_1,\Sigma_2\rt)$. 
	
	Next we wish to estimate the right-hand side of (\ref{eq:divcurl2}). Note that the two terms are symmetrical so it is 
	enough to estimate $\norm{E}_{L^{p'}}\norm{\dv B}_{W^{-1,p}}$. We start with $\norm{\dv B}_{W^{-1,p}}$. Given any test function $\zeta\in C^\infty_c(\Omega')$ it holds
	\begin{equation}
		\label{eqand20}
		-\int_{\Omega}B\cdot\na \zeta\,dx =\int_{\Omega} \na \chi \cdot D^h \Sigma_2 (m)\, \zeta \,dx
		+ \int_{\Omega}\chi \dv (D^h \Sigma_2 (m))\, \zeta \,dx.
	\end{equation}
	Now estimating each term we have 
	\begin{align*}
		\int_{\Omega}\chi \dv (D^h \Sigma_2 (m))\, \zeta \,dx & = \int_{\Omega} \dv\Sigma_2(m)\,D^{-h}(\chi\zeta)\,dx \nn\\
		&\leq \norm{\dv\Sigma_2(m)}_{L^p(\Omega'')}\norm{D^{-h}(\chi\zeta)}_{L^{p'}(\Omega)}\nn\\
		& \lesssim  \norm{\dv\Sigma_2(m)}_{L^p(\Omega'')}\norm{\chi}_{C^1}\norm{\zeta}_{W^{1,p'}(\Omega')}\abs{h},\nn
	\end{align*}
	which shows that
	\begin{equation}
		\label{eqand10}
		\norm{\chi\dv \lt(D^h\Sigma_2(m)\rt)}_{W^{-1,p}(\Omega')}\lesssim \norm{\chi}_{C^1}\norm{\dv\Sigma_2(m)}_{L^p(\Omega'')}\abs{h}.
	\end{equation}
	The other term of the right-hand side of \eqref{eqand20} can be estimated similarly:
	\begin{align*}
		\int_{\Omega} \na \chi \cdot D^h \Sigma_2 (m)\, \zeta \,dx & = \int_{\Omega} \Sigma_2\lt(m\rt)\cdot D^{-h}(\zeta\nabla\chi)\,dx\nn\\
		&\lesssim \norm{\Sigma_2(m)}_{L^\infty(\Omega)}\norm{\nabla(\zeta\nabla\chi)}_{L^{p'}(\Omega)}\abs{h}\nn\\
		&\lesssim \norm{\chi}_{C^2}\norm{\zeta}_{W^{1,p'}(\Omega')}\abs{h},
	\end{align*}
	which gives
	\begin{equation}
		\label{eqand11}
		\norm{\na\chi\cdot D^h\Sigma_2(m)}_{W^{-1,p}(\Omega')}\lesssim \norm{\chi}_{C^2}\abs{h}.
	\end{equation}
	Thus 
	\begin{equation*}
		\|\dv B\|_{W^{-1,p}(\Omega')}\overset{(\ref{eqand20}),(\ref{eqand10}), (\ref{eqand11})}{\lesssim} \lt(1+\norm{\dv\Sigma_2(m)}_{L^p(\Omega'')}  \rt)\norm{\chi}_{C^2}\abs{h}.
	\end{equation*}
	Gathering the above, and using the notation $\Sigma=(\Sigma_1,\Sigma_2)$, we obtain from \eqref{eq:divcurl2} that
	\begin{align}\label{eq:divcurl_estim2}
		\int_{\Omega} \chi^2 \abs{D^h m}^4\,dx &\lesssim pp'\norm{\chi}_{C^2}\norm{\chi D^h\Sigma(m)}_{L^{p'}(\Omega)}\lt(1+\norm{\dv\Sigma(m)}_{L^p(\Omega'')}\rt)\abs{h}.
	\end{align}
	Let us take $p=\frac{4}{3}$ and hence $p'=4$. Since $\chi\leq 1$ we have $\lt|\chi\rt|^4\leq \lt|\chi\rt|^2$. Further since $\Sigma$ is smooth we have that $\lt|D^h \Sigma\lt(m\rt)\rt|\lesssim \lt| D^h  m\rt|$.
	Thus
	\begin{equation*}
		\lt|\chi(x) D^h \Sigma\lt(m(x) \rt) \rt|^4\lesssim \lt(\chi(x) \rt)^2 \lt|D^h m(x)  \rt|^4\qd\text{ for a.e. } x
	\end{equation*}
	and it follows that
	\begin{equation*}
		\|\chi D^h \Sigma\lt(m \rt) \|_{L^4(\Omega)}\lesssim \lt(\int_\Omega \chi^2 \lt|D^h m  \rt|^4 dx\rt)^{\frac{1}{4}}.
	\end{equation*}
	Applying this to (\ref{eq:divcurl_estim2}), dividing through and using $\chi\equiv 1$ on $\overline U$ gives 
	\begin{align*}
		\lt(\int_{U} \abs{D^h m}^4\,dx\rt)^{\frac 3 4}&\lesssim \norm{\chi}_{C^2}\lt(1+\norm{\dv\Sigma(m)}_{L^{\frac 43}(\Omega'')}\rt)\abs{h}.
	\end{align*}
	This shows that $m\in B^{\frac 13}_{4,\infty}(U)$ for all $U\subset\subset\Omega$ and thus completes the proof.
\end{proof}

\begin{proof}[Proof of Theorem \ref{C0}]
	If \eqref{eqhyp1} holds, then we have  $m\in B^{\frac{1}{3}}_{3r,\infty,\loc}(\Omega)$ for $r=\min\{p, \frac 43\}$ from Theorem \ref{C2}. On the other hand, if \eqref{eqhyp2} holds, then $m\in B^{\frac{1}{3}}_{4,\infty,\loc}(\Omega)$ by Proposition \ref{LLD1}. Thus, under the assumptions \eqref{eqhyp1} or \eqref{eqhyp2}, we have $m\in B^{\frac{1}{3}}_{3r,\infty,\loc}(\Omega)$ for $r=\min\{p, \frac 43\}$. From Lemma \ref{lemp6}, we know \eqref{cpeq1.23} holds in $L^r_{\loc}(\Omega)$, and thus it follows from Theorem \ref{C1} that \eqref{cpeq2} and \eqref{abeqa1.5} hold. Further, under the assumption \eqref{eqhyp2}, we have $\dv\Phi_f(m)\in L^p_{\loc}(\Omega)$ because of the formula \eqref{cpeq2}. This completes the proof.
\end{proof}

%
%

\appendix

\section{Computations needed in the proof of Proposition~\ref{p:harmext}}
\label{a:compharmext}

In this appendix we check that \eqref{eq:harmext} holds for $f=f^j_k$ given in \eqref{f_k}, $j=1,2$, $k\geq 2$.

\begin{lem}\label{l:Phifk}Let $k\geq 2$.
	For $f=f^j_k$, the entropies $\Phi_f$ defined in \eqref{eqplm1} satisfy
	\begin{align*}
		\Phi_{f^1_k}(e^{it})&=\frac{i\cos(k\frac\pi 2)}{k}\left[
		\frac{e^{i(k+1)t}}{k+1}
		+ \frac{e^{-i(k-1)t}}{k-1}
		\right],\\
		\Phi_{f^2_k}(e^{it})&=-\frac{2}{k}e^{it} + \frac{\cos(k\frac\pi 2)}{k}
		\left[ \frac{e^{i(k+1)t}}{k+1}-\frac{e^{-i(k-1)t}}{k-1}
		\right].
	\end{align*}
\end{lem}
\begin{proof}[Proof of Lemma~\ref{l:Phifk}]
	For $f=f^1_k$ we have
	\begin{align*}
		\psi_f(t)&=\int_0^t\cos(ks)\,ds =\frac{1}{k}\sin(kt),\\
		\varphi_f(t)&=\frac 1k \int_0^t \sin(ks)ie^{is}\,ds 
		=\frac{1}{2k}\int_0^t (e^{iks}-e^{-iks})e^{is}\, ds\\
		&=\frac{1}{2k}\int_0^t(e^{i(k+1)s}-e^{-i(k-1)s})\, ds
		=\frac{1}{2k}\left[\frac{e^{i(k+1)t}-1}{i(k+1)} + \frac{e^{-i(k-1)t}-1}{i(k-1)}
		\right],
	\end{align*}
	\begin{align*}
		\Phi_f(e^{it})&=-i\varphi_f(t-\frac\pi 2) +i\varphi_f(t+\frac\pi 2)\\
		&= \frac{1}{2k}\left[
		\frac{e^{i(k+1)t}}{k+1}
		(-e^{-i(k+1)\frac\pi 2}+e^{i(k+1)\frac\pi 2}) 
		+ \frac{e^{-i(k-1)t}}{k-1}
		(-e^{i(k-1)\frac\pi 2} +e^{-i(k-1)\frac\pi 2})
		\right] \\
		& =  \frac{i\cos(k\frac\pi 2)}{k}\left[
		\frac{e^{i(k+1)t}}{k+1}
		+ \frac{e^{-i(k-1)t}}{k-1}
		\right].
	\end{align*}
	For $f=f^2_k$ we have
	\begin{align*}
		\psi_f(t)&=\int_0^t\sin(ks)\, ds = \frac 1k (1-\cos(kt)),\\
		\varphi_f(t)&=\frac{1}{k}\int_0^t(1-\cos(ks))ie^{is}\,ds 
		= \frac{i}{2k}\int_0^t (2e^{is}-e^{i(k+1)s}-e^{-i(k-1)s})\, ds \\
		& =\frac{1}{2k}\left[
		2(e^{it}-1) -\frac{e^{i(k+1)t}-1 }{ k+1} 
		+ \frac{e^{-i(k-1)t}-1}{ k-1}
		\right],
	\end{align*}
	\begin{align*}
		\Phi_f(e^{it})&=-i\varphi_f(t-\frac\pi 2) +i\varphi_f(t+\frac\pi 2)\\
		&=\frac{i}{2k}\Bigg[
		2e^{it}(-e^{-i\frac\pi 2}+e^{i\frac\pi 2})\\
		&\quad -\frac{e^{i(k+1)t}}{k+1}(-e^{-i(k+1)\frac\pi 2}+e^{i(k+1)\frac\pi 2}) + \frac{e^{-i(k-1)t}}{k-1}(-e^{i(k-1)\frac\pi 2}+e^{-i(k-1)\frac\pi 2})
		\Bigg] \\
		& = -\frac{2}{k}e^{it} + \frac{\cos(k\frac\pi 2)}{k}
		\left[ \frac{e^{i(k+1)t}}{k+1}-\frac{e^{-i(k-1)t}}{k-1}
		\right].
	\end{align*}
\end{proof}

\begin{lem}\label{l:extPhifk}
	Let $k\geq 2$.
	For $f=f^j_k$, the harmonic entropies $\Phi^{\xi_f}$ defined in \eqref{eq:extPhif} satisfy
	\begin{align*}
		\Phi^{\xi_{f^1_{k}}}(e^{it})&=\frac{i\cos(k\frac\pi 2)}{k}\left[
		\frac{e^{i(k+1)t}}{k+1}
		+ \frac{e^{-i(k-1)t}}{k-1}
		\right]  = \Phi_{f^1_k}(e^{it}),\\
		\Phi^{\xi_{f^2_{k}}}(e^{it})&=\frac{\cos(k\frac\pi 2)}{k}
		\left[ \frac{e^{i(k+1)t}}{k+1}-\frac{e^{-i(k-1)t}}{k-1}
		\right] = \Phi_{f^2_k}(e^{it}) + \frac{2}{k}e^{it}.
	\end{align*}
\end{lem}

\begin{proof}[Proof of Lemma~\ref{l:extPhifk}]
	Recall that $\Phi^\xi$ is given by
	\begin{equation}
		\label{eqad11}
		\Phi^\xi(z)=\xi(z)z+((iz)\cdot \nabla\xi(z))iz.
	\end{equation}
	For $f=f^j_k$ we have
	\begin{align}
		\label{eqad12}
		\xi_{f^j_k}&= 0\qquad\text{if }k\text{ is odd},\nn\\
		\xi_{f^1_{2k}}&=\frac{(-1)^{k+1}}{k(1-2k)(1+2k)}\varphi^1_{2k},\nn\\
		\xi_{f^2_{2k}}&=\frac{(-1)^{k}}{k(1-2k)(1+2k)}\varphi^2_{2k},
	\end{align}
	where $\varphi^j_k$ are the harmonic polynomials given in polar coordinates by  $\varphi^1_k=r^k\sin(k\theta)$, $\varphi^2_k=r^k\cos(k\theta)$. Hence
	\begin{align}
		\label{eqad13}
		\Phi^{\xi_{f^j_k}}&\overset{\eqref{eqad11}, \eqref{eqad12}}{=}0\qquad\text{if }k\text{ is odd},\nn\\
		\Phi^{\xi_{f^1_{2k}}}& \overset{\eqref{eqad11}, \eqref{eqad12}}{=}\frac{(-1)^{k+1}}{k(1-2k)(1+2k)}\Phi^{\varphi^1_{2k}},\nn\\
		\Phi^{\xi_{f^2_{2k}}}&\overset{\eqref{eqad11}, \eqref{eqad12}}{=} \frac{(-1)^{k}}{k(1-2k)(1+2k)}\Phi^{\varphi^2_{2k}},
	\end{align}
	and it remains to compute $\Phi^{\varphi^j_{2k}}(e^{it})$.
	
	For $\xi=\varphi^1_k$ we find
	\begin{align}
		\label{eqad14}
		\Phi^{\varphi^1_k}(e^{it})&=\varphi^1_k e^{it} + (-\sin t \,\partial_1\varphi^1_k + \cos t\,\partial_2\varphi^1_k)ie^{it} \nn\\
		&\overset{\eqref{appeqd3}}{=}\sin(kt)e^{it}+(-k \sin t \sin((k-1)t) +k\cos t\cos((k-1)t))ie^{it}\nn\\
		&=\sin(kt)e^{it} + k \cos(kt)ie^{it} \nn\\
		&=\frac i 2 e^{it}\left[
		-e^{ikt}+e^{-ikt} + k e^{ikt}+ke^{-ikt}
		\right] \nn\\
		&=\frac i 2\left[ (k-1)e^{i(k+1)t} + (k+1)e^{-i(k-1)t}\right],
	\end{align}
	and for $\xi=\varphi^2_k$,
	\begin{align}
		\label{eqad15}
		\Phi^{\varphi^2_k}(e^{it})&=\varphi^2_k e^{it} + (-\sin t \,\partial_1\varphi^2_k + \cos t\,\partial_2\varphi^2_k)ie^{it} \nn\\
		&\overset{\eqref{appeqd3}}{=}\cos(kt)e^{it}+(-k \sin t \cos((k-1)t) - k\cos t\sin((k-1)t))ie^{it}\nn\\
		&=\cos(kt)e^{it} -k \sin(kt)ie^{it} \nn\\
		&=\frac 12 e^{it}\left[e^{ikt}+e^{-ikt}-ke^{ikt}+ke^{-ikt} \right]\nn\\
		&=\frac 12 \left[ 
		-(k-1)e^{i(k+1)t} + (k+1)e^{-i(k-1)t}
		\right].
	\end{align}
	Gathering the above we obtain
	\begin{align*}
		\Phi^{\xi_{f^1_{2k}}}&\overset{\eqref{eqad13}, \eqref{eqad14}}{=}\frac{i(-1)^k}{2k}
		\left[
		\frac{e^{i(2k+1)t}}{2k+1} + 
		\frac{e^{-i(2k-1)t}}{2k-1}
		\right],\\
		\Phi^{\xi_{f^2_{2k}}}&\overset{\eqref{eqad13}, \eqref{eqad15}}{=} \frac{(-1)^k}{2k}
		\left[
		\frac{e^{i(2k+1)t}}{2k+1}-\frac{e^{-i(2k-1)t}}{2k-1}
		\right],
	\end{align*}
	which, gathering all the cases ($k$ even or odd),  and recalling the expressions found in Lemma~\ref{l:Phifk} for $\Phi_{f^j_k}$, proves Lemma~\ref{l:extPhifk}.
\end{proof}

\section{Proof of \cite[Lemma 3.4]{GL}}
\label{a:lem3.4}

\begin{lem}[Ghiraldin-Lamy]\label{p:sigmaLp}
	Let $m:\Omega\to\R^2$ satisfy \eqref{ageq3}. Assume $\dv\Phi_f(m)\in \mathcal M(\Omega)$ for all $f\in C^0(\mathbb{R}/ 2\pi\mathbb{Z} )$. Then there exists $\sigma\in \mathcal M(\Omega\times\mathbb{R}/ 2\pi\mathbb{Z})$ such that
	\begin{equation}
		\label{eqgl1.2}
		\langle \dv\Phi_f(m),\zeta\rangle =\langle\sigma(x,s), f(s)\zeta(x)\rangle\qquad\forall f\in C^0(\mathbb{R}/ 2\pi\mathbb{Z} ),\:\zeta\in C^0_c(\Omega),
	\end{equation}
	where $\langle \cdot,\cdot\rangle$ denotes the duality between measures and functions. 
\end{lem}

\begin{proof}
	 Since the proof follows almost exactly the same lines as that of Lemma \ref{lem:sigmaLp}, we only sketch it briefly focusing on the differences. Letting $\mathcal{Z}=C^0_c(\Omega)$ and thus $\mathcal{Z}^*=\mathcal{M}(\Omega)$, exactly as in the proof of Lemma \ref{lem:sigmaLp}, the Banach-Steinhaus' uniform boundedness principle implies the map
	\begin{align}
	\label{fineqa2}
		T\colon C^0(\mathbb{R}/ 2\pi\mathbb{Z} )\to \mathcal{Z}^*,\quad f\mapsto \dv\Phi_f(m)
	\end{align}
	is a bounded linear operator. 
	
	Again by \cite[Theorem VI.2.1]{diest}, there exists an $\mathcal M(\Omega)^{**}$-valued Borel measure $G$ on $\mathbb{R}/ 2\pi\mathbb{Z} $ such that
	\begin{equation}
		\label{fineqa3}
		\langle \varphi^*,Tf\rangle =\int_{\mathbb{R}/ 2\pi\mathbb{Z} } f\, d \langle G,\varphi^*\rangle\qquad\forall \varphi^*\in\mathcal{Z}^{**}.
	\end{equation}
	As in the proof of Lemma \ref{lem:sigmaLp}, for $\psi$ a finite linear combination of the form
	\begin{align*}
		\psi(x,s)=\sum_j \one_{E_j}(s)\zeta_j(x),\qquad E_j\subset \mathbb{R}/ 2\pi\mathbb{Z} \text{ Borelian, }\zeta_j\in \mathcal{Z},
	\end{align*}
	we define
	\begin{align}
		\label{eqokl1.2}
		\langle\sigma,\psi\rangle :=\sum_j \langle G(E_j),\zeta_j^*\rangle,\qquad\zeta_j^*:=( \mu\mapsto \langle \mu,\zeta_j\rangle)\in\mathcal{Z}^{**}.
	\end{align}
	Thus, assuming the $E_j$'s are disjoint and non-negligible, as in \eqref{eqp2}, we have
	\begin{align*}
		\langle\sigma,\psi\rangle \leq \norm{T}_{\mathcal L(C^0(\mathbb{R}/ 2\pi\mathbb{Z} );\mathcal{Z}^*)}\norm{\psi}_{L^\infty(\Omega\times\mathbb{R}/ 2\pi\mathbb{Z})}.
	\end{align*}
	After extension, $\sigma$ can therefore be considered as a continuous linear form on $C^0_c(\Omega\times\mathbb{R}/ 2\pi\mathbb{Z})$. Further, for $f=\one_{E}$ and $\zeta\in \mathcal{Z}$, we deduce from (\ref{eqokl1.2}) that
	\begin{align*}
		\langle \sigma(x,s),f(s)\zeta(x)\rangle = \int f \, d\langle G ,\zeta^*\rangle\overset{(\ref{fineqa2}),(\ref{fineqa3})}{=}\la \dv \Phi_f(m),\zeta\ra
	\end{align*}
	and this formula must also be valid for $f\in C^0(\mathbb{R}/ 2\pi\mathbb{Z} )$. This establishes \eqref{eqgl1.2}.
\end{proof}

\section{Proof of Lemma \ref{slice}}
\label{a:sigmaLp}

\begin{proof}[Proof of Lemma \ref{slice}]
	Given any open set $U\subset\subset\Omega$, recall from \eqref{eqgl1.2} we have
	\begin{equation}
		\label{eq30}
		\la \dv \Phi_f(m), \zeta \ra=\iint_{U\times \R/ 2\pi \mathbb{Z}} f(s)\zeta(x)\, d \sigma(x,s)\qquad\forall f\in C^0\lt(\mathbb{R}/ 2\pi \mathbb{Z}\rt)  ,\:\zeta\in C^{0}_c(U). 
	\end{equation}
	Let $\{f_j\}_j\subset  C^0\lt(\R/ 2\pi \mathbb{Z}\rt)$ be a countable dense subset of $C^0\lt(\R/ 2\pi \mathbb{Z}\rt)$, and $\Lambda$ be the intersection of the Lebesgue points of $\dv\Phi_{f_j}(m)$ and $\mu_m$. Note that $|\Omega\setminus\Lambda|=0$. For any   $x_0\in\Lambda\cap U$, define
	\begin{equation}
		\label{eq:L_x}
		L_{x_0}(f_j):=\lim_{r\rightarrow 0} \Xint-_{B_r(x_0)}\dv \Phi_{f_j}(m)\,dx=\dv \Phi_{f_j}(m)(x_0).
	\end{equation}
	By construction \eqref{eqplm1}--\eqref{eqplm1.5}, it is clear that $\Phi_f$ depends linearly on $f$ and $\|\Phi_f\|_{C^2}\leq C\|f\|_{C^0}$ for some constant $C>0$ independent of $f$. As $x_0$ is a Lebesgue point of $\dv\Phi_{f_j}(m)$ and $\mu_m$, it follows that
	\begin{align}
		\label{control:dvPhi_f}
		\lt|\dv \Phi_{f_j}(m)(x_0)\rt|&\leq\limsup_{r\rightarrow 0} \Xint-_{B_r(x_0)}\lt|\dv \Phi_{f_j}(m)(x)\rt|\,dx\nn\\
		&\leq C\lVert f_j\rVert_{C^0\lt(\mathbb{R}/ 2\pi\mathbb{Z}\rt)}\lim_{r\rightarrow 0} \Xint-_{B_r(x_0)}\mu_m(x)\,dx = C\lVert f_j\rVert_{C^0\lt(\mathbb{R}/ 2\pi\mathbb{Z}\rt)}\mu_m(x_0).
	\end{align}
	Thus we deduce from \eqref{eq:L_x} and \eqref{control:dvPhi_f} that $\lt|L_{x_0}(f_j)\rt|\leq  C\lVert f_j\rVert_{C^0\lt(\mathbb{R}/ 2\pi\mathbb{Z}\rt)}\mu_m(x_0)$ for all $j$. As $\{f_j\}$ is dense in $C^0\lt(\R/ 2\pi \mathbb{Z}\rt)$, $L_{x_0}$ can be uniquely extended to a bounded linear operator on $C^0\lt(\R/ 2\pi \mathbb{Z}\rt)$ satisfying
	\begin{equation}
		\label{control:L_x}
		\lt|L_{x_0}(f)\rt|\leq  C\lVert f\rVert_{C^0\lt(\mathbb{R}/ 2\pi\mathbb{Z}\rt)}\mu_m(x_0)\qd\text{for all }f\in C^0\lt(\R/ 2\pi \mathbb{Z}\rt).
	\end{equation}
	By the Riesz Representation Theorem, $L_{x_0}$ can be represented as a measure, denoted by $\sigma_{x_0}$ for all $x_0\in\Lambda\cap U$, i.e. for a.e. $x_0\in U$. Further, as a consequence of \eqref{control:L_x}, we have $\lVert \sigma_x\rVert_{\mathcal{M}(\R/ 2\pi \mathbb{Z})}\leq C\mu_m(x)$ for a.e. $x\in U$. This establishes \eqref{eq:sigmaLp} as $\mu_m\in L^p(U)$ and $U\subset\subset\Omega$ is arbitrary. 
	
	Finally, for any $f\in C^0\lt(\R/ 2\pi \mathbb{Z}\rt)$, by density of $\{f_j\}$ in $C^0\lt(\R/ 2\pi \mathbb{Z}\rt)$, there exists $f_{j_k}\to f$ in $C^0\lt(\R/ 2\pi \mathbb{Z}\rt)$. It follows from \eqref{eq:L_x} and \eqref{eq30} that
	\begin{align}
		\label{eq:slicing}
		\int_{U} \lt(\int_{ \R/ 2\pi \mathbb{Z}} f_{j_k}(s) d\sigma_x(s)\rt) \zeta(x) dx &\overset{(\ref{eq:L_x})}{=} \int_{U} \dv\Phi_{f_{j_k}}(m)(x) \zeta(x) dx\nn\\
		&\overset{(\ref{eq30})}{=}\iint_{U\times \R/ 2\pi \mathbb{Z}} f_{j_k}(s)\zeta(x) d \sigma(x,s)\qd\forall \zeta\in C^0_c(U)
	\end{align}
	for all $k$.  We can pass to the limit as $k\to\infty$ in the above right-hand side. For the left-hand side, note that $\int_{ \R/ 2\pi \mathbb{Z}} f_{j_k}(s) d\sigma_x(s)\to\int_{ \R/ 2\pi \mathbb{Z}} f(s) d\sigma_x(s)$ for a.e. $x\in U$ and $\lt|\int_{ \R/ 2\pi \mathbb{Z}} f_{j_k}(s) d\sigma_x(s)\rt|\overset{\eqref{control:L_x}}{\leq} C\|f\|_{C^0(\R/ 2\pi \mathbb{Z})} \mu_m(x)\in L^p(U)$. So by the Dominated Convergence Theorem, we can also pass to the limit in the left-hand side of \eqref{eq:slicing} to deduce \eqref{eq30.2}.
\end{proof}

\bibliographystyle{alpha}
\bibliography{aviles_giga}

\end{document}